\documentclass{article}
\usepackage[utf8]{inputenc}

\usepackage[margin=4.5cm, marginparsep=0.75cm]{geometry}

\usepackage{amsmath} 
\usepackage{amsthm}
\usepackage{amssymb}
\usepackage{authblk}
\usepackage{thm-restate}
\usepackage{eucal}
\usepackage{needspace}

\usepackage[percent]{overpic} 

\usepackage{thmtools}
\usepackage{thm-restate}
\usepackage[colorlinks=true]{hyperref}

\hypersetup{
  pdftitle = {A quasi-optimal upper bound for induced paths in sparse graphs},
  pdfauthor = {Basile Couëtoux, Oscar Defrain, Jean-Florent Raymond},
  colorlinks = true,
  linkcolor = black!30!red,
  citecolor = black!30!green
}

\renewenvironment{abstract}
{\small\vspace{-1em}
\begin{center}
\bfseries\abstractname\vspace{-.5em}\vspace{0pt}
\end{center}
\list{}{
\setlength{\leftmargin}{0.6in}%
\setlength{\rightmargin}{\leftmargin}}%
\item\relax}
{\endlist}

\newtheorem{theorem}{Theorem}[section]
\newtheorem{proposition}[theorem]{Proposition}
\newtheorem{definition}[theorem]{Definition}
\newtheorem{lemma}[theorem]{Lemma}
\newtheorem*{lemma*}{Lemma}
\newtheorem{corollary}[theorem]{Corollary}

\newtheorem{question}[theorem]{Question}

\newtheorem{remark}[theorem]{Remark}

\newtheorem{observation}[theorem]{Observation}

\usepackage[textsize=small, color=black!5]{todonotes}

\NewDocumentCommand{\odtodo}{O{} m}{\todo[color=green!90, #1]{#2}} 
\NewDocumentCommand{\odinfo}{O{} m}{\todo[color=green!2, #1]{#2}}

\NewDocumentCommand{\bctodo}{O{} m}{\todo[color=orange!20, #1]{#2}} 
\NewDocumentCommand{\bcinfo}{O{} m}{\todo[color=orange!5, #1]{#2}}

\NewDocumentCommand{\jftodo}{O{} m}{\todo[color=pink!80, #1]{#2}} 
\NewDocumentCommand{\jfinfo}{O{} m}{\todo[color=pink!50, #1]{#2}}

\reversemarginpar
\newcommand{\mgnt}[1]{\marginpar{\raggedleft \textcolor{black!70}{\emph{#1}}}}
\newcommand{\new}[1]{\emph{#1}\mgnt{#1}}

\usepackage{enumitem}

\setlist{ 
  listparindent=\parindent,
  parsep=0pt,
}

\newcommand{\ct}{{\sf T}} 
\newcommand{\st}{{\sf S}} 
\newcommand{\rt}{{\sf R}} 

\newcommand{\N}{\mathbb{N}}
\newcommand{\intv}[2]{\left \{ #1,\dots, #2\right \}}

\newcommand{\depth}{{\sf depth}}

\newcommand{\zone}{{\sf zone}}
\newcommand{\range}{{\sf range}}

\newcommand{\ancestor}{{\sf zone}\text{-}{\sf ancestor}}
\newcommand{\trace}{{\sf trace}}

\newcommand{\prefix}[2]{{#1}{#2}}
\newcommand{\suffix}[2]{{#2}{#1}}
\newcommand{\subpath}[3]{{#2}{#1}{#3}}

\newcommand{\rank}{{\sf rank}}

\DeclareMathOperator{\reach}{{\sf reach}}

\DeclareMathOperator{\lip}{\sf lip}
\DeclareMathOperator{\lp}{\sf lp}
\newcommand{\ord}{\sqsubseteq}

\newcommand{\triangleseq}{{\sf triangle}\text{-}{\sf seq}}
\newcommand{\specialseq}{{\sf special}\text{-}{\sf seq}}

\newcommand{\casestylei}[1]{\textbf{#1}}
\newcommand{\casestyleii}[1]{\textcolor{black!70}{\textbf{#1}}}
\newcommand{\casestyleiii}[1]{\textcolor{black!60}{\textbf{#1}}}

\renewcommand{\next}{\textsf{next}}

\title{
A quasi-optimal upper bound for induced\\ paths in sparse graphs}
\date{July 30, 2025}

\author[1]{Basile Couëtoux}
\author[1]{Oscar Defrain}
\author[2]{Jean-Florent Raymond}

\affil[1]{Aix-Marseille Université, CNRS, LIS, Marseille, France.}
\affil[2]{CNRS, ENS de Lyon, Université Claude Bernard Lyon 1, LIP, UMR5668, Lyon, France}

\begin{document}

\maketitle

\begin{abstract}
    In 2012, Ne\v{s}et\v{r}il and Ossona de Mendez proved that graphs of bounded degeneracy that have a path of order $n$ also have an induced path of order $\Omega(\log \log n)$.
    In this paper we give an almost matching upper bound by describing, for arbitrarily large values of~$n$, 2-degenerate graphs that have a path of order $n$ and where the longest induced paths have order $O((\log \log n)^{1+o(1)})$.

    \vskip5pt\noindent{}{\bf Keywords:} long induced paths, unavoidable induced subgraphs, sparse graphs, degeneracy.
\end{abstract}

\section{Introduction}

Ramsey's Theorem, proved in 1930, is a cornerstone of extremal combinatorics. In in its finite
form, it states that every large enough graph excluding a given clique contains a large independent set.
There are several analogues of this result where the desired objects are not independent sets but other induced substructures and where the considered graphs should satisfy some additional properties. A simple example is the following theorem for connected graphs.
\begin{theorem}[see {\cite[(5.3)]{ding1996unavoidable}}%
]
    For every integer $t \geq 1$, there is an increasing function $f$ such that every connected $K_t$-free graph on $n$ vertices contains as induced subgraph $K_{1, f(n)}$ or $P_{f(n)}$.
\end{theorem}

There are also further extensions for other substructures than induced subgraphs and various requirements of connectivity, see \cite{allred2022unavoidable} and the references therein.
In the 80's, Galvin, Rival, and Sands proved the following Ramsey-type result for \emph{traceable graphs}, i.e., graphs that have a Hamiltonian path.

\begin{theorem}[{\cite[Theorem 4]{galvin1982ramsey}}]\label{th:grs}
For every integer $t\geq 1$, there is an unbounded function $h_t$ such that every $K_{t,t}$-subgraph free traceable graph on $n$ vertices contains an induced path of order at least $h_t(n)$.
\end{theorem}

This can also be restated as follows: \textit{if a $K_{t,t}$-subgraph free graph has a path of order $n$, then it has an induced path of order at least $h_t(n)$}.
Note that cliques and bicliques do not satisfy such a statement: they have a path visiting all their vertices yet no induced path of order more than 2 or 3, respectively. Hence, for hereditary graph classes, being $K_{t,t}$-subgraph free is necessary in the above statement.\footnote{If a graph contains a large biclique as a subgraph, it contains either a large biclique as induced subgraph, or a large clique (by Ramsey's theorem).}

The question motivating this work is to find good bounds on the function $h_t$ of Theorem~\ref{th:grs}.
In addition to being interesting on its own, this question was asked in \cite{galvin1982ramsey} already in the special case $t=2$.
As additional motivation, let us describe two other research lines that would benefit from progress on this question.
The first one is algorithmic. Theorem~\ref{th:grs} was rediscovered thirty years later in \cite{atminas2012linear} (with worse bounds) and used to design a parameterized algorithm for the Biclique problem. Any improved bound on the growth rate of $h_t$ would result in an improved complexity bound.
The second one is related to treedepth, an important structural graph parameter in the theory of sparse graphs.\footnote{Because we will not use treedepth in this paper, we refrain from providing the definition and refer the interested reader to \cite[Chapter~6]{nevsetvril2012sparsity} instead.} It is well-known that a graph with no path of order $n$ has treedepth at most $n$ (see, e.g.,  \cite[Proposition~6.1]{nevsetvril2012sparsity}). Using Theorem~\ref{th:grs} we can deduce the following induced analogue: a $K_{t,t}$-subgraph free graph with no induced path of order $h_t(n)$ has treedepth at most~$n$. Here again improving bounds on $h_t(n)$ would strengthen the statement.

In \cite{galvin1982ramsey}, the authors of Theorem~\ref{th:grs} did not provide an explicit expression of the function $h_t$, but it can be extracted from their proof and the currently best-known bounds on the multicolor Ramsey numbers that $h_t(n) = \Omega_t((\log \log \log n)^{1/3})$.
This result has recently been improved to $h_t(n) = \Omega_t((\log \log n)^{1/5})$ by Duron, Esperet, and the last author \cite{duron2024long}. Very recently, Hunter, Milojević, Sudakov, and Tomon gave in \cite{hunter2024long} an elegant and short proof that \[h_t(n) = \Omega_t(\log \log n / \log\log \log n).\] At the time of writing this is the best-known lower-bound on $h_t$ and no matching upper-bound is known.

A related question was asked in the \textit{Sparsity} textbook \cite{nevsetvril2012sparsity}. Let us say that a graph $G$ is \emph{$k$-degenerate} if every subgraph of $G$ has a vertex of degree at most~$k$. Note that $k$-degenerate graphs are $K_{k+1,k+1}$-subgraph free, hence they are covered by the results above. In the aforementioned book, 
Ne\v{s}et\v{r}il and Ossona de Mendez gave a simple yet clever proof of the following statement. (Unless otherwise specified, all logarithms are binary.)

\begin{theorem}[\cite{nevsetvril2012sparsity}]\label{th:nodm}
Let $k \in \N$ and let $G$ be a $k$-degenerate graph. If $G$ has a path of order $n$, then $G$ has an induced path of order at least
$
{\log \log n}/{\log (k+1)}.
$ 
\end{theorem}

The also explicitly asked the question of finding the optimal bound.

\begin{question}[{\cite[Problem 6.1]{nevsetvril2012sparsity}}]\label{que:nodm}
For every $k\in \N$, what is the maximum function $f_k\colon \N\to \N$ such that every $k$-degenerate graph $G$ that has a path of order $n$ has an induced path of order at least $f_k(n)$?
\end{question}

More generally, for a graph class $\mathcal{C}$ we can define $f_{\mathcal{C}}$ as the maximum function such that, for every graph $G\in \mathcal{C}$ and every $n$, if $G$ has a path of order $n$ then $G$ has an induced path of order $f_{\mathcal{C}}(n)$. Several recent papers study the growth rate of $f_{\mathcal{C}}$ for various choices of $\mathcal{C}$ such as planar graphs or graphs of bounded treewidth, see \cite{duron2024long} for an overview.

Since \cite{nevsetvril2012sparsity} appeared, no improvement over the  lower bound of Theorem~\ref{th:nodm} was made. On the other hand, in 2023 the second and the third author gave the following upper bound: there are 2-degenerate graphs that have a path of order $n$ and where all induced paths have order $O((\log \log n)^2)$ \cite{defrain2024sparse}. Observe that 2-degenerate graphs are also $K_{t,t}$-subgraph free for every $t\geq 3$, hence this also provides an upper bound on the function of Theorem~\ref{th:grs}.

The main result of this paper is the following.

\begin{restatable}{theorem}{restatemaintheorem}
\label{thm:main-theorem}
    There is a constant $c$ such that for infinitely many integers $n$, there is a 2-degenerate graph $G$ that has a path of order $n$ and no induced path of order more than $c \cdot \log \log n \cdot \log \log \log n$. 
\end{restatable} 

This is an improvement over the $O((\log \log n)^2)$ bound of \cite{defrain2024sparse} and answers Question~\ref{que:nodm} up to a $\log\log \log n$ factor. Note that 2-degenerate graphs are also $k$-degenerate for any $k\geq 3$ hence the above result is indeed an upper bound on the funtion of Theorem~\ref{th:nodm}, for any value of $k$.  As a consequence, the bounds on the functions of the aforementioned theorems become:
\begin{itemize}
    \item for every integer $t\geq 3$, there is a constant $c_t > 0$ such that for every integer $n \geq 16$ the function $h_t$ of Theorem~\ref{th:grs} satisfies
    \[
    c_t \cdot \frac{\log \log n}{\log \log \log n} \leq h_t(n) \leq c \cdot \log \log n \cdot \log \log \log n;
    \]
    \item for every integers $k\geq 1$ and $n \geq 4$, the function $f_k$ of Question~\ref{que:nodm} satisfies:
    \[
    \frac{\log \log n}{\log (k+1)} \leq f_k(n) \leq c \cdot \log \log n \cdot \log \log \log n,
    \]
\end{itemize}
where $c$ is the constant from Theorem~\ref{thm:main-theorem}.

\paragraph{Our techniques.}
Our proof of Theorem~\ref{thm:main-theorem} is inspired by \cite{defrain2024sparse} but our construction and the arguments are much more involved. Let us give an overview of its different steps.

The first step is the explicit construction of the graph whose existence is posited by Theorem~\ref{thm:main-theorem}. This graph should be 2-degenerate, contain a long path $P$ and no long induced paths.
Observe that vertices outside of $P$ are not useful to show any of the desired properties.  So actually our construction will not contain such vertices: the path $P$ will span all the vertices, i.e., it will be a Hamiltonian path.
More precisely, for every integer $\ell$, we construct a $2$-degenerate graph $G$ of order $\smash{\Omega\big(2^{2^\ell}\big)}$ that admits a Hamiltonian path and in which all induced paths have order $O(\ell \log \ell)$.
The different steps of the construction can be sketched as follows (see Figure~\ref{fig:summary} for an overview):
%
\begin{enumerate}
    \item We start with a complete binary tree on $\smash{2^{2^\ell-1}}-1$ nodes, and that we call the \emph{skeleton-tree}. This tree is obviously $2$-degenerate, has the desired number of vertices, but there are two drawbacks: it has no Hamiltonian path and its longest induced paths are too long (of order $\Theta(2^\ell))$.
    \item This skeleton-tree is transformed into a a graph called \emph{ribbed-tree} obtained by adding a number of \emph{barriers}.
    Informally, barriers are paths replacing the edges of the skeleton-tree together with a set of edges called \emph{ribs} that connect nodes of such paths to certain well-chosen ancestor nodes in the tree. 
    The role of these barriers is to block induced paths. They are defined in a way so that the obtained graph remains 2-degenerate.
    Their definition uses an auxiliary tree called \emph{index-tree} and certain walks on this tree. 
    While it can be shown that the induced paths have now order $O(\ell \log \ell)$ as desired (for an appropriate choice of the index-tree), the obtained graph does not have a Hamiltonian path.

    \item Finally, we apply on the ribbed-tree an operation called \emph{blow-up} that results in a graph $G$ that contains a Hamiltonian path, remains $2$-degenerate, has $\smash{\Omega(2^{2^\ell})}$ vertices (hence a path of the same order), and where induced paths have order ${O(\ell \log \ell)}$, as desired.
\end{enumerate}

\begin{figure}[ht]
    \centering
    \begin{overpic}[scale=1, grid=false]{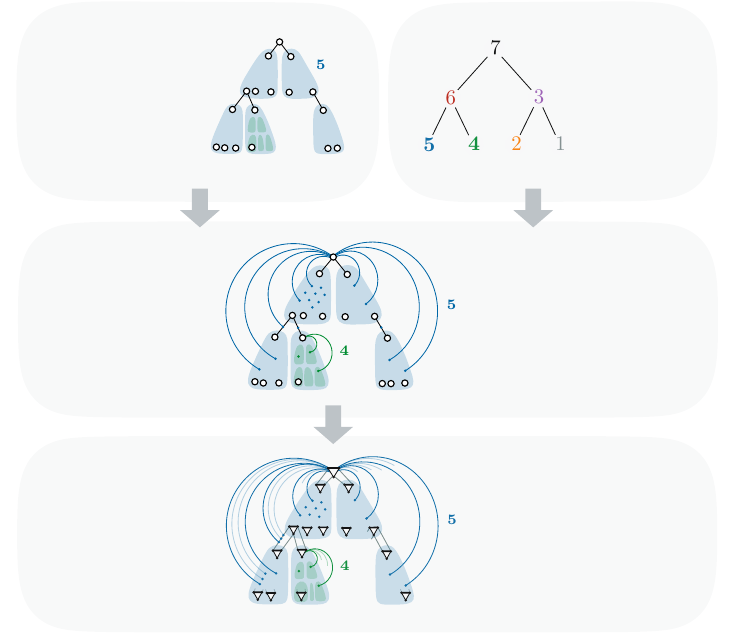}
        \put(4,72.5){\parbox{80pt}{\raggedleft
            Skeleton-tree $\st_\ell$\\[0.3em]
            $|V(\st_\ell)|=2^{2^\ell-1}-1$\\[0.3em]
            $\lip(\st_\ell)=O(2^\ell)$\\[0.3em]
            Section~\ref{sec:skeleton-tree}
        }}
        \put(79,72.5){\parbox{4cm}{
            Index-tree $\ct$\\[0.3em]
            $|V(\ct)|=\ell$\\[0.3em]
            Section~\ref{sec:index-tree}
        }}
        \put(66,43){\parbox{4cm}{
            Ribbed-tree $\rt(\ct)$\\[0.3em]
            $\smash{|V(\rt(\ct))|=\Omega\big(2^{2^\ell}\big)}$\\[0.3em]
            Section~\ref{sec:ribbed-tree}
        }}
        \put(66,14){\parbox{4cm}{
            Blow-up $G(\ct)$\\[0.3em]
            $\smash{\lp(G(\ct))=\Omega\big(2^{2^\ell}\big)}$\\[0.3em]
            $\lip(G(\ct))=O(\ell \log \ell)$\\[0.3em]
            2-degenerate\\[0.3em]
            Section~\ref{sec:blow-up}
        }}
    \end{overpic}
    \caption{The different steps in the construction of the graph satisfying Theorem~\ref{thm:main-theorem}, detailed in Section~\ref{sec:construction}, starting from an arbitrary integer $\ell$. }
    \label{fig:summary}
\end{figure}
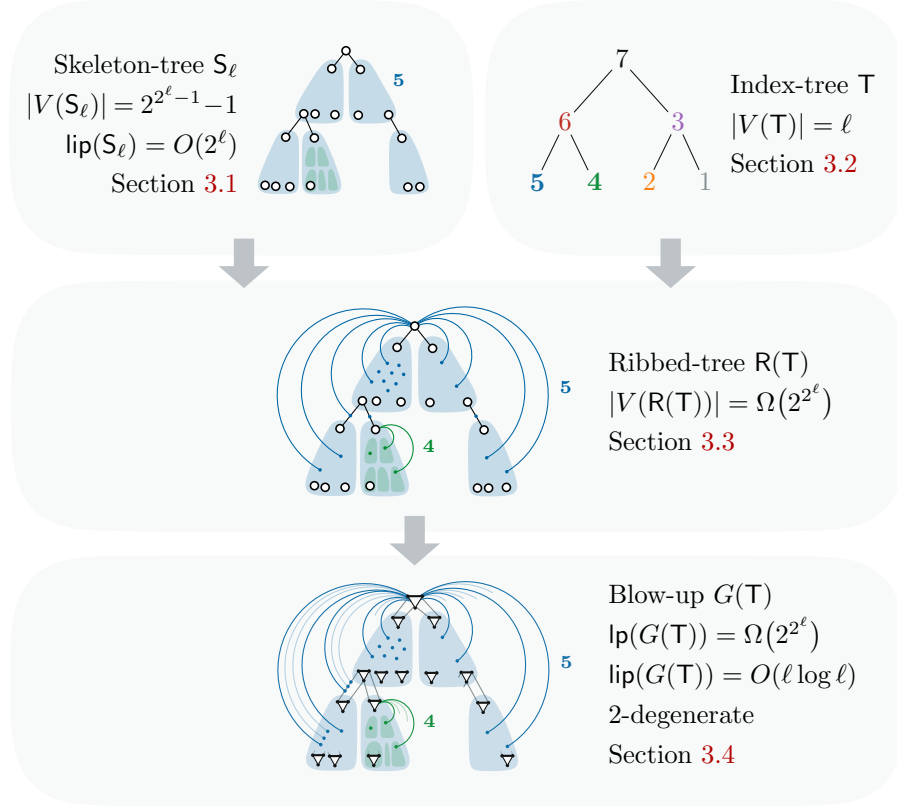

How do we show that $G$ does not have long induced paths? This is the crux of the proof and it relies on the properties of the auxiliary structures mentioned above.
The approach is to define a notion of \emph{rank} of a vertex and to show that when we follow an induced path, the vertices it uses are more and more constrained according to the ranks of the vertices the path has visited so far. 
More precisely, the ranks are in bijection with the vertices of the index-tree and we show that, as the path goes along, the ranks of the vertices it visits are further nested in the index-tree, thanks to the barriers.
This results into a bound on the orders of induced paths which roughly equals the maximum size of a barrier. This number depends on the choice of the index-tree.

Finally, the last step of the proof consists in bounding the size of barriers by a function of $\ell$. When the index-tree is a complete binary tree, the barriers have size $O(\ell^{1.6})$.
We show that the index-tree can be chosen so that barriers have size $O(\ell \log \ell)$. This value is in fact the result of a trade-off between the maximum size of a left subtree, and the maximum depth of the tree.
This proves the desired upper bound on the order of induced paths in $G$.

\paragraph{Organization of the paper.}
The rest of the paper is organized as follows.
In Section~\ref{sec:prelim} we introduce the standard notations that we use in this paper.
In Section~\ref{sec:construction} we detail the construction of our graph following the different steps highlighted in Figure~\ref{fig:summary}.
At the end of this section is argued that the obtained graph contains a Hamiltonian path, and that it is $2$-degenerate.
In Section~\ref{sec:path-properties} we characterize the shape of particular induced paths in this graph, and show that their order is function of the barrier's size.
In Section~\ref{sec:barrier-size} we upper bound the barrier's size, and prove Theorem~\ref{thm:main-theorem} in Section~\ref{sec:proof}.

\section{Preliminaries}\label{sec:prelim}

\paragraph{Words.} A \new{word} over an alphabet $\Sigma$ is an ordered sequence of elements of $\Sigma$, that are called \emph{letters}\mgnt{letter}.
A \new{factor} of a word $W$ is a contiguous subsequence of $W$.
An \emph{internal}\mgnt{internal, $\circ$} letter of a word is a letter distinct from its extremities.
We use $\circ$ for word concatenation. 
If $W$ is a word over $\Sigma$, the \new{collapse} of $W$ is the word obtained by contraction of consecutive repeated elements of $W$, i.e., each time we have two consecutive symbols in $W$, we remove one.

\paragraph{Graphs and trees.}
We assume that the reader is familiar with standard graph theory terminology, and only recall the different types of trees we consider in this paper, which are defined as follows.
\begin{itemize}
    \item A \new{rooted tree} is a tree with a distinguished node called its \new{root}.
    The \new{leaves} are the nodes with degree one and different from the root. The other vertices (including the root) are called \new{internal nodes}. For every pair $s, t$ of nodes of a rooted tree we write $s \preceq t$ if $s$ lies on the unique path connecting $t$ to the root, and $s \prec t$ if in addition $s\neq t$.
    We say that $t$ is a \new{descendant} of $s$ and that $s$ is an \new{ancestor} of $t$ whenever $s \preceq t$; $t$ is a \new{child} of $s$ and $s$ is the \new{parent} of $t$ if in addition $t$ and $s$ are neighbors.

    \item A \new{binary tree} is a rooted tree in which every internal node has exactly two children, one being referred to as the \emph{left child}\mgnt{left/right child} of the node, and the other as its \emph{right child}.
    
    \item A \new{complete binary tree} is a binary tree where the distances from the root to each leaf are identical.
    
    \item A \new{partial binary tree} is a subtree of a binary tree, i.e., its internal nodes may not have a left or right child.
\end{itemize}

The \new{depth} of a node $s$ in a rooted tree is the order of the unique path from $s$ to the root, denoted $\depth(s)$; note in particular that the root has depth $1$.
The \emph{depth} of a rooted tree $T$, denoted $\depth(T)$, is the maximum order of a path from the root to the leaf. 

\section{The construction}\label{sec:construction}

We describe the construction of a family of graphs satisfying Theorem~\ref{thm:main-theorem}.
Namely, for each integer $\ell$, we show how to construct a $2$-degenerate graph $G$ with $\smash{\Omega(2^{2^\ell})}$ vertices, containing a Hamiltonian path, and no induced path of order more than
$c\cdot \ell \log \ell$
for some integer $c$.

The construction of $G$ consists of different steps which are split into several
sections, each involving an auxiliary structure (see Figure~\ref{fig:summary}).
First, an underlying tree structure called \emph{skeleton-tree} is described in Section~\ref{sec:skeleton-tree}.
This tree is extended into a \emph{ribbed-tree} by integrating a number of ``barriers'' aimed at blocking induced paths in Section~\ref{sec:ribbed-tree}.
This later notion of barrier involves trees and intricate sequences of integers that are described before, in Section~\ref{sec:index-tree}.
Finally, the obtained graph is modified to make it traceable and $2$-degenerate in Section~\ref{sec:blow-up}.

\subsection{Skeleton-trees}\label{sec:skeleton-tree}

\newcommand{\skn}{\ell} 

We inductively define, for every positive integer $\skn$, a tree $\st_\skn$ that is a base object of our construction and that we call \new{skeleton-tree}. In the remaining of the paper, and to avoid confusions with the other structures we will consider, we will exclusively refer to the elements of $V(\st_\ell)$ as \emph{nodes}. Note that $\st_\skn$ is nothing more than a complete binary tree. 
However the definition below will allow us to define additional terminology: the \new{rank} and the \new{zone} of a node $t \in V(\st_\skn)$, which are respectively an integer in $\intv{1}{\skn}$ and a subgraph of $\st_\skn$ that we denote by $\rank(t)$ and $\zone(t)$.

\begin{itemize}
    \item The tree $\st_1$ is the single-vertex graph, and its root is its unique vertex.
    The rank of this vertex is $1$ and its zone is $\st_1$ itself.
    \item For every $\skn\geq 2$, the tree $\st_\skn$ is obtained as follows.
    \begin{enumerate}
        \item We start from a node $r$, which will be the root of $\st_\skn$, two disjoint copies of $\st_{\skn-1}$ of respective roots $r_1$ and $r_2$, and make $r$ adjacent to $r_1$, its left child, and to $r_2$, its right child.
        \item Then, for \emph{every} leaf $t$ of the obtained tree, we create two additional disjoint copies of $\st_{\skn-1}$ of respective roots $r'_1$ and $r'_2$, and make $t$ adjacent to $r'_1$, its left child, and to $r'_2$, its right child.
    \end{enumerate}    
    This completes the construction of $\st_\skn$.
    The rank of $r$ is $\skn$, its zone is $\st_\skn$, and the ranks and zones of the other nodes are those that were assigned to them in the copy of $\st_{\skn-1}$ they belong to.
\end{itemize}

For a node $t$ of $\st_\skn$, the function \new{$\ancestor_t$} maps any integer $i$ with $\rank(t) < i \leq \skn$ to the unique ancestor $s$ of $t$ such that $\rank(s) = i$ and $t\in \zone(s)$.\footnote{Note that in general $t$ might have several ancestors of rank $i$, but only one whose zone contains~$t$.}
Moreover, we define \mgnt{$\zone_t$}$\zone_t(i):=\zone(\ancestor_t(i))$ if $i\leq \ell$.
See Figure~\ref{fig:skeleton-tree} for an illustration of this terminology.

\begin{figure}[h!]
    \centering
    \includegraphics[scale=1]{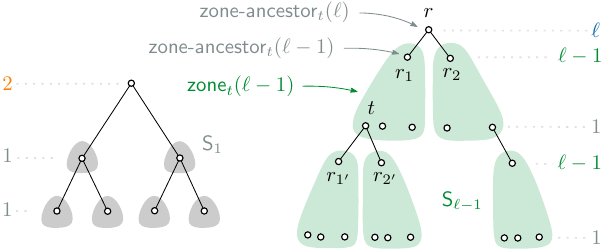}
    \caption{Inductive construction of a skeleton tree ($\st_2$ on the left, $\st_\skn$ on the right) and related terminology. Dotted lines denote the ranks of the nodes lying at the same depth in the tree. Zones are depicted in shaded areas.}
    \label{fig:skeleton-tree}
\end{figure}

Remark that if $s=\ancestor_t(i)$ then $s$ is of rank $i$ and $\ancestor_{s}$ coincides with $\ancestor_{t}$ for all values greater than $i$. 

The following simple observation will be important in the proof.   
\begin{remark}\label{rem:diffij}
Adjacent nodes in $\st_\skn$ are either of consecutive rank, or one of the two has rank $1$. 
\end{remark}

\begin{remark}\label{rem:skeleton-tree-size}
By construction we have:
\[
\left \{
\begin{array}{l}
     \depth(\st_{1})=1\\
     \depth(\st_{\skn})=2\cdot\depth(\st_{\skn-1})+1,\ \text{for any}\ \skn \geq 2. 
\end{array}
\right.
\]
A trivial induction shows $\depth(\st_{\skn})=2^{\skn}-1$, hence
\begin{equation}
    |V(\st_{\skn})|=2^{2^\skn -1}-1.
\end{equation}
\end{remark}

\subsection{Index-trees and index-barriers}\label{sec:index-tree}

In this section we define particular sequences of integers (called index-barriers) that will be used to define a notion of ``barrier'' for induced paths.
We also prove a number of properties related to index-barriers that will be crucial in our proof.

For an integer $\ell$, an \new{$\ell$-index-tree} is a partial binary (possibly but non-necessarily complete) tree $\ct$ whose vertices are the integers $1,\dots,\ell$, called \emph{indices}\footnote{In order to avoid any confusion between index-trees and the skeleton-trees of the previous section, we will exclusively refer to the vertices of $\ct$ as \emph{indices}.}\mgnt{index} arranged as the reverse of the left-child-depth-first search\footnote{That is, Depth First Search where we consider the left child of a node before its rigth child.} initiated at its root.
See Figure~\ref{fig:index-tree} for an illustration.
Then, its root is $\ell$ and the indices that appear in the left subtree of an index $i$ are all larger to the indices that appear in its right subtree, for any $i \in \intv{1}{\ell}$. 

Let us fix an $\ell$-index-tree $\ct$.
Given an index $i$ of $\ct$ we denote by $i^-$ its left child, and by $i^+$ its right child;\mgnt{$i^-$, $i^+$} the subtree rooted at $i$ is denoted by $\ct(i)$,\mgnt{$\ct(i)$} and by $\ct(i^-)$ and $\ct(i^+)$ are thus meant the subtrees rooted at $i^-$ and $i^+$, respectively.
If the left or right child does not exist these notations refer to the empty set and the empty graph.

\begin{figure}[h!]
    \centering
    \includegraphics[scale=1]{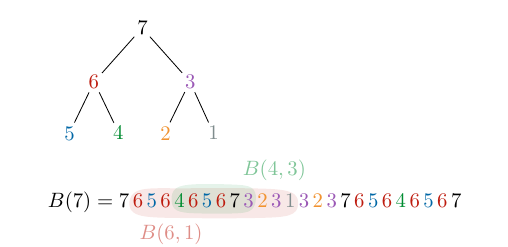}
    \caption{A 7-index-tree $\ct$, its associated full index-barrier $B(7)$, and two index-barriers $B(6,1)$ and $B(4,3)$. For better readability, to each index is associated a distinct color. Recall that index-trees are not necessarily complete binary trees.}
    \label{fig:index-tree}
\end{figure}

The following notion will be central in our construction.

\begin{definition}\label{def:full-barrier}
For any index $i\in \intv{1}{\ell}$, the \new{full index-barrier} of $i$, denoted $B(i)$, is the word over the alphabet $\{1, \dots, i\}$ defined by:\mgnt{$B(i)$}
\begin{itemize}
    \item $B(i)= i$ if $i$ is a leaf of $\ct$;
    \item $B(i)= i \circ B(i-1)\circ i$ if $i$ has only one child; and
    \item $B(i)= i\circ B(i^-)\circ i\circ B(i^+)\circ i\circ B(i^-)\circ i$ otherwise.
\end{itemize}
\end{definition}

Alternatively, note that $B(i)$ can be seen as a walk on the index-tree $\ct$ that is recursively defined similarly to a DFS except that we visit twice the left subtree. That is, to explore the subtree of $\ct$ rooted at index $i$, we visit $i$, then explore the left subtree, then visit $i$, then explore the right subtree, then visit $i$, then explore the left subtree again, then visit $i$. 

This alternative but equivalent point of view will be of interest later, in particular to make the following basic observations as well as to prove a number of properties on factors of $B(\ell)$.

\begin{observation}\label{obs:descendant-smaller}
    For any $i$ with two children, the integers in $B(i^+)$ are smaller than those in $B(i^-)$, which in turn are smaller than $i$.
\end{observation}

\begin{observation}\label{obs:factors-contain-sp}
    If two indices $i,j \in \intv{1}{\ell}$ appear in some factor $W$ of $B(\ell)$, then every index in the shortest $i$--$j$ path in $\ct$ appears in $W$ as well.
\end{observation}


\begin{observation}\label{obs:decreasing-indices}
    Let $k,i,j\in \intv{1}{\ell}$ be such that $k>i>j$ and $i,j\in V(\ct(k))$.
    Then the first occurrence of $i$ appears before the first occurrence of $j$ in $B(k)$. 
\end{observation}

\def\ancestorvariable{k}

\begin{proposition}\label{prop:smaller-index-incomparable-right-child}
    Let $i,b\in \{1,\dots,\ell\}$.
    If $i<b$ then either $i$ lies in $\ct(b^-)$, or it lies in $\ct(\ancestorvariable^+)$ for some ancestor $\ancestorvariable$ of $b$ ($\ancestorvariable=b$). 
\end{proposition}

\begin{proof}
    Consider the (unique) path $Q$ of $\ct$ from the root $\ell$ to $b$.
    By Observation~\ref{obs:descendant-smaller} each index in $Q$ is larger than $b$.
    Moreover, if a index $\ancestorvariable$ in $Q$ has its left child $\ancestorvariable^-$ outside the path, then $b \in V(\ct(\ancestorvariable^+))$ and by the same observation every index in $\ct(\ancestorvariable^-)$ is larger than $b$.
    
    Thus, the indices of $\ct$ that are smaller than $b$ either lie in $\ct(b^-)$, or in $\ct(\ancestorvariable^+)$ for some index $\ancestorvariable$ of $Q$, with possibly $\ancestorvariable=b$.
\end{proof}

\begin{proposition}\label{prop:from-parent-of-root-to-i-consecutive}
    Let $i,j\in \{1,\dots,\ell\}$ be such that $j\in \ct(i)$ (in particular $j\leq i$), and such that $i$ has a parent $i'$ in $\ct$.
    Then any factor of $B(\ell)$ from an occurrence of $i'$ to an occurrence of $j$ contains all indices from $i$ to $j$. 
\end{proposition}

\begin{proof}
    Consider a factor $W$ of $B(\ell)$ from an occurrence of $i'$ to an occurrence of~$j$.   
    Without loss of generality we may assume that none of $i'$ and $j$ appears as an internal letter of $W$ (otherwise we may consider the shorter factor using that letter).
    Then by definition of a full index-barrier, $W$ is a prefix of $i' \circ B(i) \circ i'$. Applying Observation~\ref{obs:decreasing-indices}
    with $i'$ instead of $k$, and every $j' \in \intv{i}{j}$ instead of $i$, we obtain that all integers between $i$ and $j+1$ appear before the first occurrence of~$j$, hence in $W$, as claimed.
\end{proof}

In the following, we will be interested in some factors of $B(\ell)$ that will be used to define ``barriers'' for induced paths in our final construction.
These particular factors are defined as follows.

\begin{definition}\label{def:ij-barrier}
    For\mgnt{index-barrier} 
    any $i,j \in \intv{1}{\ell}$ with $i>j$, the \emph{index-barrier} from $i$ to $j$, denoted $B(i,j)$,\mgnt{$B(i,j)$} is the factor of $B(\ell)$ starting with the first occurrence of $i$ and ending at the first occurrence of $j$.
\end{definition}

Note that the index-barrier $B(i,j)$ is defined for $i>j$ only.
It contains one occurrence of $j$ and possibly several occurrences of $i$.
Moreover, by Observation~\ref{obs:decreasing-indices}, the occurrence of $j$ in $B(i,j)$ is in fact the first occurrence of $j$ in $B(\ell)$.
In the remaining of the paper, two particular types of index-barriers will be considered, namely, those of the form $B(i-1,i)$ and $B(i,1)$ for $i \geq 2$.

We derive the following observations from the fact that index-barriers are factors of the full index-barrier $B(\ell)$, using Observation~\ref{obs:decreasing-indices} with $\ell$ in place of $k$.

\begin{observation}\label{obs:concatenation-of-barriers}
    The collapse of $B(i_1,i_2)\circ B(i_2,i_3)\circ \cdots \circ B(i_{p-1},i_p)$ where $\ell\geq i_1> i_2> \cdots> i_p\geq 1$ is a factor of $B(\ell)$.
\end{observation}

From the definition and Proposition~\ref{prop:from-parent-of-root-to-i-consecutive} we also have the following.

\begin{observation}\label{obs:i-j-intermediate}
    For any $i,j \in \intv{1}{\ell}$ with $i>j$, $B(i,j)$ contains all indices from $i$ to $j$, no internal occurrence of $j$, and it does not contain indices smaller than $j$.
    In particular, any prefix of $B(i,1)$ ending at the first occurrence of $j$ satisfies these properties.
\end{observation}

We include a proof for the next observation, which relies on the position of consecutive indices in the index-tree.

\begin{observation}\label{obs:i-iminusone-intermediate}
    For any $i \in \intv{2}{\ell}$, $B(i,i-1)$ only contains indices greater or equal to $i-1$, and it does not contain any internal occurrence of $i$ nor $i-1$.
\end{observation}

\begin{proof}
    The fact that $B(i,i-1)$ does not contain internal occurrences of $i-1$ follows from the definition, and the fact that it only contains indices greater than $i-1$ is implied by Observation~\ref{obs:i-j-intermediate}.
    Let us show that no internal occurrence of $i$ appears in $B(i,i-1)$.
    
    We consider two cases depending on whether $i$ is a leaf of $\ct$ or not.

    Suppose that $i$ is not a leaf. 
    All occurrences of $i$ appears in a factor $B(i)$ by definition. 
    The first occurrence of $i$ is therefore in $B(i)$. 
    If $i$ has two children, $i-1$ is its left child, and $B(i)=i\circ B(i-1)\circ i\circ B(i^+)\circ i\circ B(i-1)\circ i$ and $B(i-1)$ start with $i-1$. Therefore $B(i,i-1)= i\circ i-1$. Similarly, if $i$ has one child, $B(i)=i\circ B(i-1)\circ i$ and therefore $B(i,i-1)= i\circ i-1$.
    
    Suppose that $i$ is a leaf.
    To deal with this case, we first need to prove the following statement: for every $j \in \intv{1}{\ell}$, if an index $i'$ is the smallest that appears in $B(j)$, then it appears exactly once. 
    This can be shown by the following induction on the depth of $j$.
    The statement is true when $j$ is leaf as in this case $B(j)=j=i'$.
    If $j$ has two children then $i'$ is in $\ct(j^+)$ (because the indices in $\ct(j^+)$ are smaller than those in $\ct(j^-)$). Since all indices of $\ct(j^+)$ appear in $B(j)$, we get that $i'$ is the smallest index of $T(j^+)$. Therefore, by induction, it appears once in $B(j^+)$, and so only once in $B(j)=j\circ B(j^-)\circ j\circ B(j^+)\circ j\circ B(j^-)\circ j$.
    Similarly, if $j$ has one child then by definition $B(j)=j\circ B(j-1)\circ j$ and by applying induction on $j-1$ we get that the smallest index appears only once.
    
    Now let $j$ be the smallest index such that $i$ and $i-1$ belong to $\ct(j)$ (that is, $j$ is the least common ancertor of $i$ and $i-1$).
    By construction of $\ct$, $i-1$ is not an ancestor of $i$ and since $i$ has no descendant, $j$ has two children.
    Since $i$ is a leaf, $i\in \ct(j^-)$ and $i-1$ is the right child of $j$. 
    Since $i$ is the smallest index of $\ct(j^-)$ it appears once in $B(j^-)$.
    So $B(j)=j\circ B(j^-)\circ j\circ B(i-1)\circ j\circ B(j^-)\circ j$. Since $B(i,i-1)$ is a factor of $B(j)$ and since $B(j^-)$ contains only one occurrence of $i$, so does $B(i,i+1)$.
\end{proof}


\begin{proposition}\label{prop:factors-bounded-by-b-are-equal}
    Let $i,b \in \intv{1}{\ell}$ be such that $i<b$.
    Then all the factors of $B(\ell)$ starting with $b$, ending with $b$, containing at least one occurrence of $i$, and no other occurrence of $b$, are equal.
\end{proposition}

\begin{proof}
    By Proposition \ref{prop:smaller-index-incomparable-right-child} either $i$ lies in $\ct(b^-)$, or it lies in $\ct(\ancestorvariable^+)$ for some ancestor $\ancestorvariable$ of $b$ with possibly $\ancestorvariable=b$.
    Let $W$ be a factor of $B(\ell)$ satisfying the conditions of the statement.
    We distinguish the two cases.
        
    Suppose that $i\in \ct(b^-)$. 
    By construction every $i$ in $B(\ell)$ in fact appears in a factor of the form $B(b^-)$, and every such occurrence of $B(b^-)$ does not contain the letter $b$ and is preceded and followed by $b$ in $B(\ell)$.
    Hence we have $W=b\circ B(b^-)\circ b$.
    
    Otherwise let $\ancestorvariable$ be the ancestor of $b$ (with possibly $\ancestorvariable=b$) such that $i\in \ct(\ancestorvariable^+)$. 
    Again, note that by construction each occurrence of $i$ in $B(\ell)$ lies in a factor of the form $B(\ancestorvariable)$.
    We argue that $W$ must be factor of $B(\ancestorvariable)$. 
    Indeed, note that if $W$ is not a factor of $B(\ancestorvariable)$ then by Observation~\ref{obs:factors-contain-sp} it contains the parent of $\ancestorvariable$.
       
    But then, by construction, $W$ contains either a suffix or a prefix of
    \[
    B(\ancestorvariable)=\ancestorvariable\circ B(\ancestorvariable^-)\circ \ancestorvariable\circ B(\ancestorvariable^+)\circ \ancestorvariable \circ B(\ancestorvariable^-)\circ \ancestorvariable.
    \]
    As $i$ appears only in $B(\ancestorvariable^+)$
    and $b$ appears in $k\circ B(\ancestorvariable^-)$, we get that $b$ is an internal letter of $W$, a contradiction.
    
    So $W$ is indeed a factor of $B(k)$. From the definition of $B(k)$ as recalled just above, we conclude that $W$ starts with the last occurrence of $b$ in the prefix $\ancestorvariable\circ B(\ancestorvariable^-)\circ \ancestorvariable$ and ends with the first occurrence of $b$ in the suffix $\ancestorvariable\circ B(\ancestorvariable^-)\circ \ancestorvariable$.

    In both cases, $W$ is uniquely characterized, as desired.
\end{proof}

The following relation will be at the core of our proof.

\begin{definition}
    For every $b \in \intv{1}{\ell+1}$, the relation \new{$\reach_b$} is defined on integers $i,j$ smaller than $b$ as follows: $i \reach_b j$ if there is a factor of $B(\ell)$ containing $i$ and $j$ but not $b$.
\end{definition}

Clearly the relation $\reach_b$ is reflexive and symmetric. 
Proposition~\ref{prop:factors-bounded-by-b-are-equal} above implies that it is transitive. Hence it is an equivalence relation.
Note that we allow $b=\ell+1$, in which case we trivially have $i \reach_b j$ for all $i,j\in \intv{1}{\ell}$.

\begin{proposition}\label{prop:same-subtree}
    Let $b\in \{1,\dots,\ell\}$ and let $i,j<b$ be such that $i \reach_b j$.
    Then $i$ and $j$ lie either both in $\ct(b^-)$, or both in $\ct(\ancestorvariable^+)$ for some ancestor $\ancestorvariable$ of $b$ with possibly $\ancestorvariable=b$. 
\end{proposition}

\begin{proof}
    By Proposition~\ref{prop:smaller-index-incomparable-right-child}, since $i<b$, $i$ either lies in $\ct(b^-)$, or in $\ct(\ancestorvariable_1^+)$ for some ancestor $\ancestorvariable_1$ of $b$ with possibly $\ancestorvariable_1=b$.
    Since $j<b$, we obtain analogously that $j$ lies in $\ct(b^-)$, or in $\ct(\ancestorvariable_2^+)$ for an ancestor $\ancestorvariable_2$ of $b$.

    If $i,j \in \ct(b^-)$ we are done. Below we consider the other cases.
    Consider a minimum factor $W$ of $B(\ell)$ containing $i$ and $j$ but avoiding $b$, which exists since $i \reach_b j$. 
    By Observation~\ref{obs:factors-contain-sp}, this factor contains the indices of the shortest path connecting $\ancestorvariable_1$ and $\ancestorvariable_2$ in $\ct$.

    In the case where $i\in \ct(b^-)$ and $j \in \ct(\ancestorvariable_2^+)$ for some ancestor $\ancestorvariable_2$ of $b$ (and the symmetric case where the roles of $i$ and $j$ are reversed), $W$ contains $b$ since it lies on the $i$-$j$ path of $\ct$, a contradiction.

    In the case where $i \in V(\ct(\ancestorvariable_1^+))$ and $j \in V(\ct(\ancestorvariable_2^+))$ for some ancestors $\ancestorvariable_1^+, \ancestorvariable_2^+$ of $b$, we may assume by symmetry that $\ancestorvariable_1^+$ is lower than $\ancestorvariable_2^+$. Then $b \in V(\ct(\ancestorvariable_1^-)$. If $\ancestorvariable_1 \neq \ancestorvariable_2$, by definition of $B(\ell)$ the factor $W$ contains all the indices of $\ct(\ancestorvariable_1^-)$, including $b$, a contradiction.
    Hence $\ancestorvariable_1=\ancestorvariable_2$ and the statement follows.
\end{proof}

In the remaining of the proof, we will be interested in restrictions of the reachable indices we define as 
\begin{align*}
    \reach^+_b(i)&= \{j : j \reach_b i\ \text{and}\ j>i\},\\
    \reach^-_b(i)&= \{j : j \reach_b i\ \text{and}\ j<i\}.
\end{align*}

\subsection{Ribbed-tree and node-barriers}\label{sec:ribbed-tree}

So far we described two different types of objects: skeleton-trees and index-trees. In this section we explain how to construct a graph called \emph{ribbed-tree}, obtained by modifying a skeleton-tree according to a given index-tree.
These modifications consist in adding edges playing the aforementioned role of ``barriers'' for induced paths, in the sense that they will prevent the later from extending too long. 

Let us consider, for some integer $\ell$, an $\ell$-index-tree $\ct$. Note the number of indices in $\ct$ equals the number of ranks in $\st_\ell$.
For every edge of $\st_\ell$ we perform modifications that we describe now.

Let $s,t$ be two adjacent nodes of $\st_\ell$, with  $s\prec t$, and let $i,j$ be their respective ranks.
Then either $i\geq 2$, in which case we have $j=i-1$, or $i=1$, in which case $j$ may be any index in $\{1,\dots,\ell-1\}$; see Remark~\ref{rem:diffij}.

\begin{figure}[ht]
    \centering
    \includegraphics[scale=1]{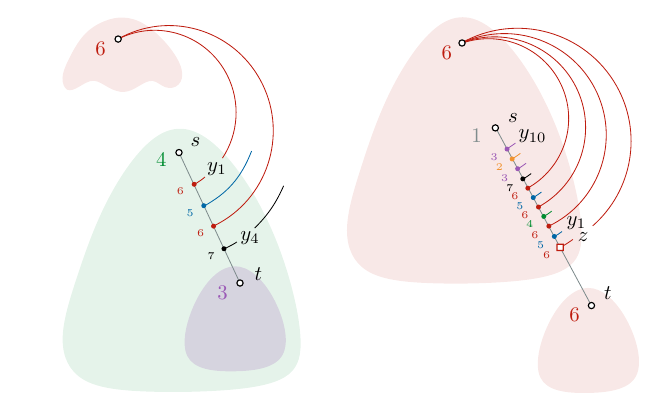}
    \caption{The construction of the barriers in the two cases where $4=i\geq 2$ (left) and $i=1$ (right), for $\ell=7$ and with $\ct$ being a complete binary tree as in Figure~\ref{fig:index-tree}. 
    The ranks of the nodes are indicated by colored integers.
    The shaded areas correspond to the zones, with the zone of $6$ partially represented on the left for better readability: it should be understood that it contains the two other represented zones as well.
    The square node (right) represents the extra subdivision $z$ in the case where $i=1$. 
    Note that the obtained subdivided nodes are considered in the zone of their parent node $s$.}
    \label{fig:node-barrier}
\end{figure}

Let us first describe the modifications to perform when $i\geq 2$.
These modifications are better understood accompanied with Figure~\ref{fig:node-barrier}.
Let $p=|B(i,j)|-2$. If $p=0$, then $i$ and $j$ are adjacent in $\ct$. 
In this case we do not modify the edge $st$.
Otherwise, we subdivide it $p$ times, where each obtained node will represent an internal index of $B(i,j)$.
Formally, let $y_1,\dots,y_p$ be the obtained subdivided vertices, in this order from $s$ to $t$. 
Let $c_1,\dots, c_{p} \in \intv{1}{\ell}$ be the indices such that $B(i,j) = i c_1\cdots c_p j$. 
For every $1\leq k\leq p$ we add the edge $y_ks_k$ where $s_k=\ancestor_s(c_k)$, i.e., we connect $y_k$ to the node of $\st_\ell$ of rank $c_k$ whose zone contains~$s$.
Note that distinct such $y_k$ may be adjacent to a same node $s_k$, if they represent a same index of the index-barrier. 
We point that these modifications are performed indifferently if $t$ is a left child of $s$, or if it is a right child.

If $i=1$ we proceed analogously as in the first case, except that we consider the index-barrier from $j$ to $i$ instead, since $j\geq i$, and subdivide the edge once more with the final node being adjacent to its ancestor of rank $j$.
More formally, we proceed as follows.
If $i=j=1$, let $p=0$.
Otherwise let $p=|B(i,j)|-2$ and note that $p\geq 0$. 
We subdivide $p+1$ times the edge $st$.
Let $z,y_1,\dots,y_p$ be the obtained vertices, in order from $t$ to $s$ this time, i.e., from bottom to top in the skeleton-tree.
If $p\geq 1$, let $c_1,\dots, c_{p} \in \intv{1}{\ell}$ be the indices such that $B(j,1) = j c_1\cdots c_p 1$.
For every $1\leq k\leq p$ we add the edge $y_ks_k$ where $s_k=\ancestor_s(c_k)$.
Then we add the edge $zs'$ for $s'=\ancestor_s(j)$. 

The graph resulting of the above process, denoted by $\rt(\ct)$, is called the \emph{ribbed-tree}\mgnt{ribbed-tree, $\rt(\ct)$} obtained from $\ct$.
We call \emph{$\ell$-ribbed-tree} any ribbed-tree obtained from an $\ell$-index-tree.

Since we only added vertices and edges, we trivially obtain the following.

\begin{remark}\label{rem:ribbed-tree-size}
    By construction and Remark~\ref{rem:skeleton-tree-size}, for every $\ell$-ribbed-tree $\rt$ we have $|V(\rt)|= \Omega\big(2^{2^\ell}\big)$. 
\end{remark}

In the remaining of the paper, we call \new{barrier-path} the resulting path obtained by subdividing an edge $st$ (including its endpoints $s$ and $t$).
We call \new{blocking-node} the internal nodes of these paths, and \new{ribs} the incident edges that connect them to their zone ancestor.
The \emph{bottom endpoint}\mgnt{bottom/top endpoint} of a rib is the node that is a blocking-node; the \emph{top endpoint} is its other endpoint.
The extra rib that is incident to the node $z$ as described above in the case where $i=1$ is called a \new{hopping rib}.

By \new{barrier} we mean a barrier-path and its incident ribs.

Note that each blocking-node $x$ is naturally associated to an index in $\{1,\dots,\ell\}$ it represents, which is the rank of the node $s$ it shares the rib with.
We extend the $\rank$ function to map every such $x$ to $\rank(s)$.
Furthermore, we extend the definition of $\zone(s)$ to include the blocking-nodes of any barrier-path between a node $s'$ in $\zone(s)$ and one of its children $t$.

\subsection{The blow-up of a ribbed-tree}\label{sec:blow-up}

In the previous section we described the concept of a ribbed-tree. The purpose of the edges added in the construction of the ribbed-tree is to ensure that there is no long induced path, because our final goal is to construct a graph where all the induced paths are short. However, we also want this graph to have a long path and to be 2-degenerate. 
Therefore in this section we describe an operation to transform a ribbed-tree into a graph that contains a Hamiltonian path and is $2$-degenerate (the proof that it is indeed the case is given in Section~\ref{sec:traceable-and-degenerate}).
This operation, that we call the \new{blow-up}, is defined as follows. 
It is similar to the ``blow-up operation'' described in \cite{defrain2024sparse} but differs on the way ribs are handled.
These modifications are better understood accompanied with Figures~\ref{fig:blowup} and \ref{fig:G3}.

Let $\ell$ be an integer, $\ct$ be an $\ell$-index-tree an let $\rt := \rt(\ct)$.
For every node $s$ of $\rt$ that is not a blocking-node we create a clique on three vertices \new{$K^s$} and choose an edge of $K^s$ to be its \new{top edge}.  
Vertices of $K^s$ are called \emph{triangle-vertices}\mgnt{triangle-vertex}.
In the top edge, we distinguish one endpoint that we call the \emph{left endpoint}\mgnt{left/right endpoint}, and the other is called the \emph{right endpoint}.
The vertex not lying in the top-edge is called \emph{bottom triangle-vertex}\mgnt{top/bottom triangle-vertex} and those of the top-edge are the (left or right) \emph{top triangle-vertices}.

For every barrier between two nodes $s,t$ in $\rt$ with $s\prec t$, we proceed as follows:
\begin{itemize}
    \item First, we start with a copy of the node-barrier, including $s$ and $t$. 
    Then we replace each blocking-node $x$ representing an index $i$ by a path $x_1x_2x_3$, where $x_1$ is adjacent to the neighbor of $x$ preceding it (when following the node-barrier from $s$ to $t$), and $x_3$ is adjacent to neighbor of $x$ following it.
    We will say that $x_1,x_2,x_3$ are the \emph{representatives}\mgnt{representative} of the occurrence of the index $i$ in the associated index-barrier.
    
    \item Recall that for every blocking-node $x$ of the barrier, a rib $e$ was connecting it to an ancestor node $s'$. 
    To represent $e$, we create three ribs $e_1,e_2,e_3$ with $x_1$ connecting $e_1$ to the bottom triangle-vertex of the triangle $K^{s'}$, $e_2$ connecting $x_2$ to the right-endpoint of the top-edge of this triangle, and $e_3$ connecting $x_3$ to the left-endpoint.
    We call \new{left barrier-path} the obtained path, and \new{left barrier} the left barrier-path together with its incident ribs.

    \item We duplicate the \emph{left barrier} to obtain a second barrier that shall be referred to as the \emph{right}\mgnt{right barrier} one.

    \item Finally, we connect the left and right barriers between triangles in the following way.
    
    If $t$ is the left (resp.~right) child of $s$, we proceed as follows.
    We identify the left (resp.~right) endpoint of the top edge of $K^s$ with the endpoint $s$ of the left (resp.~right) barrier-path, and the left (resp.~right) endpoint of the top edge of $K^t$ with the endpoint $t$ of the left (resp.~right) barrier-path. 
    Then we identify the bottom triangle-vertex of $K^s$ with the endpoint $s$ of the right (resp.~left) barrier-path, and the right (resp.~left) endpoint of the top edge of $K^t$ with the endpoint $t$ of the right (resp.~left) barrier-path. 

    Note that the differences between a left and a right child are mirrored; see Figure~\ref{fig:G3} for a global illustration of the result.
\end{itemize}

\begin{figure}[ht]
    \centering
    \includegraphics[scale=1, page=2]{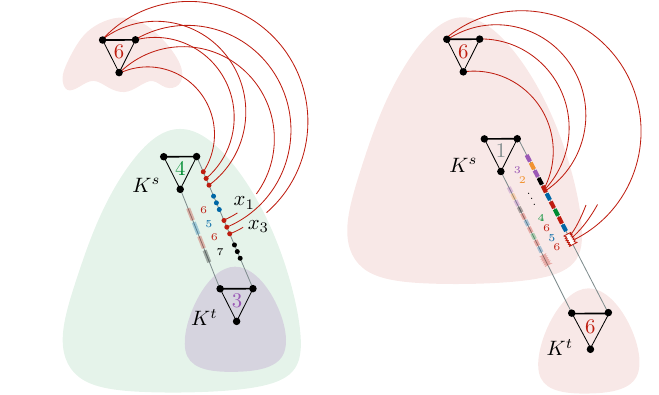}
    \caption{The blowup of a ribbed-tree detailed on the two cases where $i\geq 2$ (actually $i=4$) on the left and $i=1$ on the right, for $\ell=7$. The groups of 3 consecutive vertices $x_1,x_2,x_3$ representing an index of the index-barrier are represented by rectangles, and detailed only once in the right barrier of the case $i\geq 2$ (left).
    Spiky rectangles represent vertices incident to hopping ribs.
    Thick edges represent the top edges of the triangles.}
    \label{fig:blowup}
\end{figure}

The graph $G$ obtained by this operation is defined as the \emph{blow-up} of $\rt$. Note that $G$ is uniquely determined by the choice of $\ct$ hence we denote by \new{$G(\ct)$} the blow-up of the ribbed-tree obtained from $\ct$. A partial representation of the graph $G$ obtained as above from an index-tree $\ct$ that is a complete binary tree on three indices is given in Figure~\ref{fig:G3}.

\begin{figure}
    \centering
    \includegraphics[scale=1]{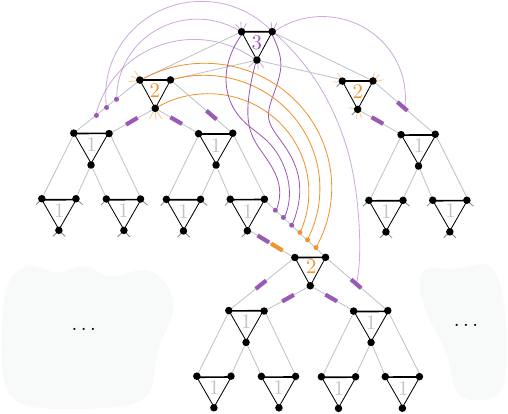}
    \caption{A partial representation of the graph $G$ where $\ct$ is a complete binary tree on three indices.
    The rank of triangle-vertices are depicted within the triangles. 
    For better readability, not all ribs are depicted; it should be understood that the triangle at the root sends ribs to any blocking-vertex of rank $3$, in purple in the picture.}
    \label{fig:G3}
\end{figure}

In the following, to avoid confusion between $G$ and the other auxiliary structures, we will exclusively refer to the elements of $V(G)$ as \emph{vertices}.
Note that for every triangle-vertex $v$ in $G$, there is a node $t$ of $\st_\ell$ such that $v$ belongs to the triangle $K^t$. 
We denote this node by \new{$\pi(v)$}.
Similarly, if $x$ is a representative then $\pi(x)$ denote the blocking node of $\rt$ it originates from.

We extend all the terminology from $\rt$ to $G$ as follows, with the additional technicality that now, to a given node $t$ corresponds a set of three vertices, namely, those in $K^t$:
\begin{itemize}
    \item $\rank(v):=\rank(\pi(v))$;

    \item $\zone(v):= \{u\in V(G) : \pi(u) \in \zone(\pi(v))\}$; and
    
    \item $\zone_v(i):= \{u\in V(G) : \pi(u) \in \zone_{\pi(v)}(i)\}$; and

    \item $\ancestor_v(i):= \{u\in V(K^s) : s=\ancestor_{\pi(v)}(i)\}$.
\end{itemize}

In particular, we will say that $K$ is a \emph{child triangle}\mgnt{child/parent triangle} of $K'$ if $t$ is the child of $s$ in $\st_\ell$ for $s$ and $t$ such that $K=K^{\pi(s)}$ and $K'=K^{\pi(t)}$; we define the \emph{parent triangle} analogously.

\begin{remark}\label{rem:graph-size}
By construction and Remark~\ref{rem:ribbed-tree-size}, for every $\ell$-index tree $\ct$ we have $|V(G(\ct))|=\Omega\left (2^{2^\ell} \right )$.
\end{remark}

We conclude the section with the next definition and observation on the sequence of ranks in barrier-paths.
Recall that the collapse of a word is the contraction of repeated consecutive entries; see Section~\ref{sec:prelim}.
The last remark of the following observation is a consequence of Observation \ref{obs:i-iminusone-intermediate} and the presence of hopping ribs.

\begin{definition}\label{def:trace}
    The \new{trace} of a sequence $v_1,\dots,v_q$ of vertices in $G$ is the word obtained by collapsing  $\rank(v_1)\circ \cdots\circ \rank(v_q)$.
\end{definition}

\begin{observation}\label{obs:barrier-trace}
    Let $Q$ be a barrier-path from a triangle-vertex $u$ to a triangle vertex $v$ (with $u$ and $v$ included) where $\pi(u)\prec \pi(v)$. 
    Let $i:=\rank(u)$ and $j:=\rank(v)$.
    Then the trace of $Q$ is equal to:
    \begin{itemize}
        \item $B(i,i-1)$, if $i\geq 2$;
        \item the reverse of $B(j,1)$, if $i=1$.
    \end{itemize}
    Moreover, if $Q'$ is the subpath of $Q$ consisting of blocking-vertices, then its trace does not contain $i$ and $i-1$ if $i\geq 2$, and it always contains $j$ if $i=1$. 
\end{observation}

\subsection{Traceability and degeneracy}\label{sec:traceable-and-degenerate}

We note that the graph $G(\ct)$ described previously only differs from the one constructed in \cite{defrain2024sparse} up to a larger number of subdivisions performed between triangles, each of which yielding a single extra rib.

Consequently, an analysis similar to the one conducted in~\cite[Lemmas~4.1 and 4.3]{defrain2024sparse} yields the two following lemmas, whose proofs are included for the sake of completeness.

\begin{lemma}\label{lem:traceable}
    For any index-tree $\ct$ the graph $G(\ct)$ contains a Hamiltonian path.
\end{lemma}
\begin{proof}
    Let $\ell = V(\ct)$ (i.e. $\ct$ is a $\ell$-index tree) and $G = G(\ct)$.  For ever node $s$ of $\st_\ell$ we denote by $G_{\downarrow s}$ the subgraph of $G$ induced by all the triangle vertices of some $K^t$ where $t$ is a descendant of $s$ in $\st_\ell$ (including the case $t=s$) as well as all barrier-paths between these triangles.

    We will prove by induction on the structure of the graph the following statement: for every node $s$ of $\st_\ell$, the graph $G_{\downarrow s}$ has a Hamiltonian path from the right to the left top triangle vertices of $K^s$.
    This implies the statement of the lemma when $s$ is the root of $S_\ell$.

    The base case is when $s$ is a leaf of $\st_\ell$. In this case $G_{\downarrow s}$ is a triangle and the desired property clearly holds.

    Suppose now that $s$ is not a leaf and that the property holds both for it left child $t_L$ and its right child $t_R$.      
    Then a Hamiltonian path of $G_{\downarrow s}$ can be obtained as follows: starting from the right top triangle vertex of $K^s$, we follow the barrier-path leading to the right top triangle vertex of $K^{t_R}$, then we follow the Hamiltonian path of $G_{\downarrow t_R}$ provided by the induction hypothesis and end up at the left top triangle vertex of $K^{t_R}$, then we follow the barrier-path leading to the bottom triangle vertex of $K^s$ followed by that leading to the right top triangle vertex of $K^{t_L}$, we follow the Hamiltonian path of $G_{\downarrow t_L}$ given by the induction hypothesis and conclude with the barrier-path ending at the left top triangle vertex of $K^s$.
    \end{proof}
\begin{lemma}\label{lem:degenerate}
    For any index-tree $\ct$ the graph $G(\ct)$ is $2$-degenerate.
\end{lemma}

\begin{proof}
    %
    Let $\ell = V(\ct)$ (i.e. $\ct$ is a $\ell$-index tree) and $G=G(\ct)$.
    For every node $s$ of $\st_\ell$ that is not the root, the barrier-paths \new{above} $K^s$ are the barrier-paths of $G$ between $K^s$ and $K^t$, where $t$ denotes the parent of $s$ in $\st_\ell$.
    We define a partial order $\ord$ on the vertices of $G(\ct)$ as follows.
    For every $u,v \in V(G(\ct))$, we have $u \ord v$ if there are some $s,t \in \st_\ell$ with $s\prec t$ such that one of the following holds:\footnote{Parentheses are added in order to avoid ambiguity about the precedence of logical operators.}
    \begin{enumerate}
        \item\label{e:sym} ($u\in K^s$ or $u$ is an internal vertex of some barrier-path above $K_s$) and ($v\in K^t$ or $v$ is in an internal vertex of some barrier-path above $K^t$), or
        \item \label{e:same} $u$ and $v$ belong to the same barrier-path between $K^s$ and $K^t$ (in whih case $t$ is a child node of $s$ in $\ct_\ell$) and $K^s$ is closer from $u$ than from $v$; or
        \item \label{e:tri} ($u,v\in K^s$ or $u,v\in K^t$) and $v$ is the bottom triangle vertex.
    \end{enumerate}    

    Let $H$ be a subgraph of $G$ and note let $u$ be a maximal vertex of $H$ with respect to~$\ord$.
    We will distinguish cases to show that $u$ has degree at most $2$ in $H$, which implies the desired statement.

    \begin{itemize}
        \item If $u$ is the interior vertex of a barrier path above $K^t$ for some $t\in \st_\ell$, it has 3 neighbors in $G$, by construction (two on the barrier-path and one via a rib). However the neighbor $v$ of $u$ on the barrier-path that is the closest to $K^t$ is such that $u \ord v$ according to item~\ref{e:same} above, hence $v \notin V(H)$ (by maximality of $u$) and $u$ indeed has at most two neighbors.
        \item If $u$ is a triangle vertex in some $K^s$ then in $G$ it is has neighbors that are endpoints of ribs, however by construction every such vertex belongs to a barrier paths above $K^t$ for some $t \in \st_\ell$ with $s\prec t$, so as above such neighbor do not appear in $H$. The same is true of neighbors in the barrier paths above $K^t$ when $t$ is a child of $s$. The possible remaining neighbors are the two other vertices of $K^s$ if $u$ is the bottom triangle vertex, or, if $u$ is a top triangle vertex, the other top triangle vertex and a vertex in the barrier path above. In both cases $u$ has degree at most 2 in $G$.
    \end{itemize}
    Recall that triangle vertices and interior vertices of barrier-path partition the vertex set of $G$, so we considered all cases. Hence $G$ is 2-degenerate.
\end{proof}

\section{Properties of induced paths}\label{sec:path-properties}

In this section we consider, for some $\ell\in \N$ and some $\ell$-index-tree $\ct$, the graph $G := G(\ct)$ and prove properties on particular induced paths it contains. The final goal is to show that $G$ does not contain long induced paths.

In the remaining of the section, let $P$ be an induced path of $G$.
Let us also fix an orientation of $P$, and further assume that its endpoints are triangle-vertices, with the first endpoint of largest rank among $P$.
For any two vertices $u$ and $v$ with $v$ after $u$ in $P$, we will use the following notations:
\begin{itemize}
    \item $\prefix{P}{v}$ is the prefix of $P$ ending in $v$; 
    \item $\suffix{P}{v}$ is the suffix of $P$ starting in $v$; 
    \item $\subpath{P}{u}{v}$ is the subpath from $u$ to $v$; and
    \item $\next(v)$ is the first triangle-vertex appearing after $v$ in $P$ and that does not belong to the same triangle as $v$.
\end{itemize}
Let us motivate the definition of the function $\next$.
Note that each triangle of $G$ may be intersected at most twice by $P$, and if it is intersected twice, then it is by two consecutive vertices of $P$.
Indeed, any other scenario would induce a chord in $P$, which is excluded.
Thus, and to avoid unnecessary technicalities, we will identify one of the two vertices to be a representative of the intersection of $P$ in each triangle.
This representative will be the first of the two vertices according to the ordering of $P$, and this is the vertex that the function $\next$ maps to.

Accordingly, in the remaining of the section, let us denote by \new{$\triangleseq(P)$} the sequence $v_1,\dots,v_p$ of vertices of $P$ where $v_1$ is the first (triangle-)vertex of $P$, $v_p$ belongs to the last triangle intersected by $P$, and for all $1\leq i< p$, $v_{i+1}=\next(v_i)$.

We start with the following observation which mainly follows by construction, and on which we will deeply rely, often implicitly.

\begin{proposition}\label{prop:index-trace-factor}
    If $u,v$ are two consecutive vertices in $\triangleseq(P)$, then the trace of $\subpath{P}{u}{v}$ is a factor of $B(\ell)$.
\end{proposition}

\begin{proof}
    Let $K:=K^{\pi(u)}$. 
    Consider the first (non necessarily triangle-)vertex $u'$ not in $K$ that follows $u$ in $P$, and $e$ be the edge incident to $u'$ that precedes it in $P$.
    By construction, either $e$ lies in a barrier separating $K$ from its parent triangle, or it lies in a barrier separating $K$ from a child triangle, or it is a rib with a vertex of $K$ for top endpoint.
    
    If $e$ lies in a barrier, then either $P$ crosses the full barrier, in which case it reaches $v$ in the attained triangle,
    or $P$ contains a rib attached to the barrier, in which case $v$ is the top endpoint of that rib.
    In the first case, by Observation~\ref{obs:barrier-trace}, the trace of $\subpath{P}{u}{v}$ is $B(i,j)$ with $i$ and $j$ being the respective ranks of $u$ and $v$, or of $v$ and $u$, depending on whether $\rank(u)>\rank(v)$ or $\rank(u)<\rank(v)$.
    Hence this trace is a factor of $B(\ell)$.
    In the other case, note that $\subpath{P}{u}{v}$ has visited a part of $B(i,j)$, starting from $\rank(u)$ to an occurrence of $\rank(v)$.
    Hence the trace of $\subpath{P}{u}{v}$ is a factor of $B(\ell)$ as desired.

    In the remaining case where $e$ is a rib, $u'$ lies in a barrier separating two triangle $K^s$ and $K^t$, whose trace is $B(i,j)$ for some $i,j\in \{1,\dots,\ell\}$.
    Then $P$ may either reach one of these two triangles, or reach another index in the barrier and take a corresponding rib.
    The same arguments as above shows that the trace of $\subpath{P}{u}{v}$ is a factor of $B(\ell)$ in these cases.
\end{proof}

Notice that if an induced path visits a triangle vertex $u$ without following the ribs attached to it, then it will not contain the bottom endpoints of duch ribs.
Additionally, we will argue that if an induced path has intersected a rib on its bottom endpoint, then it will not be able to contain other ribs of that rank within the same zone.
We formalize the conditions that should be verified by $P$ in order to verify these properties.

Recall that given a vertex $v$ of $G$, $\ancestor_v$ maps each index $i$ to $V(K^s)$ where $s$ is the unique vertex of rank $i$ such that $v\in \zone(s)$.

The following notions are defined with respect to the path $P$ we fixed.
\begin{definition}\label{def:ghost}
    An index $i\in \{1,\dots,\ell\}$ is a \new{ghost rank} if $i>b^*$ for $b^*$ the largest rank of a triangle-vertex in $P$.
\end{definition}

\begin{definition}\label{def:burn}
    A triangle-vertex $v$ of $P$ \emph{burns}\mgnt{to burn} an index $i\in \intv{1}{\ell}$ if $i$ is a ghost rank, or if $\prefix{P}{v}$ contains the bottom endpoint of a rib of index $i$ incident to some $u\in \ancestor_v(i)$.
\end{definition}

Note that in Definition~\ref{def:burn} only one rib incident to some $u\in \ancestor_v(i)$ has to be intersected. We will argue that one such rib suffices to ``block'' our induced path in barriers.
The reason for considering ghost ranks in this definition is to unify coming statements and arguments. Note that bottom-endpoint of ribs of such ranks may be visited, but the ribs will not be used by $P$, as otherwise $P$ would contain a triangle-vertex of that rank, contrary to the definition. Hence they will play the same role as other ribs whose ranks are burned.





\begin{definition}[Index-lock]
    Let $v$ be a triangle-vertex of $P$ and $u$ be the latest vertex of $\triangleseq(P)$ before $v$ such that $v\in \zone(u)\setminus K^{u}$, if any.
    If such a vertex exists, the \new{index-lock} of $v$ is defined as $\rank(u)$, otherwise we set it to $\ell+1$.
\end{definition}

Note that the index-lock of the first vertex of $P$ is $\ell+1$; this artificial value is set to unify most of the remaining arguments.
We note the following which is trivial if no $u$ as in the definition above exists, since the index-lock is $\ell+1$ in that case, and which follows by construction using the fact that $v\in \zone(u)\setminus K^{u}$ otherwise.

\begin{observation}\label{obs:index-lock-rank-decreases}
    The index-lock of a triangle-vertex $v$ is larger than its rank.
\end{observation}

Accordingly, let us extend the notation $\zone_v(i)$, i.e., the function that maps $i$ to the zone of rank $i$ containing $v$, so that it now maps to $V(G)$ when $i=\ell+1$.
Similarly, $\ancestor_v$ now maps $\ell+1$ to the root clique of $G$.

Intuitively speaking, the index-lock $i$ of a triangle-vertex is the rank of the latest zone it has strictly entered to, this zone being the full graph if $i=\ell+1$.

We are now ready to introduce the notion which will be at the core of our characterization of $P$, and which connects with index-trees from Section~\ref{sec:index-tree}.

\begin{definition}\label{def:correct}
    A triangle-vertex $v$ of $P$ of rank $a$ and index-lock $b$ is \new{correct} if every $i\in \reach^+_b(a)$ is burned by $v$.
\end{definition}

Again, note that in this definition, it is not required that $v\in \zone(u)\setminus K^{u}$ for some $u$ before $v$ in $\triangleseq(P)$. 
If no such $u$ exists then the index-lock of $v$ is $\ell+1$, and so $\reach^+_{\ell+1}(a)=\{a,\dots,\ell\}$.
In particular, note that the first vertex of $P$ is necessarily correct as it is of largest rank, and hence each of $a,\dots,\ell$ is a ghost rank in that case.
See Figure~\ref{fig:correct} for examples of correct vertices. 

\begin{figure}[htb]
\begin{minipage}[b]{0.48\textwidth}
\hspace{30pt}
\includegraphics[page=1, scale=1]{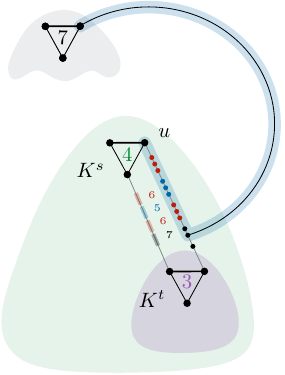}
\end{minipage}
\hspace{20pt}
\begin{minipage}[b]{0.48\textwidth}
\includegraphics[page=2, scale=1]{figures/correct.pdf}
\end{minipage}
\caption{Three possible situations for a vertex $u$ to be correct when its rank is $4$ (left) and $3$ (right); in this example it is assumed that $\ct$ is a complete binary tree as in Figure~\ref{fig:index-tree}. In red and blue are represented different possibilities for $\prefix{P}{u}$. 
In blue, the index-lock of $u$ is $7$:
note that $\prefix{P}{u}$ intersects every rib of rank in $\reach^+_7(4)=\{5,6\}$ (left), or every rib of rank in $\reach^+_7(3)=\emptyset$ (right).
In red, the index-lock of $u$ is $6$, and $7$ is a ghost rank: it is the only index in $\reach^+_6(3)$.}
\label{fig:correct}
\end{figure}

The key step of our proof is showing that every triangle-vertex of $P$ is correct, which we use to characterize the way $P$ behaves between consecutive vertices in $\triangleseq(P)$.
For this, we will rely on the following property which expresses a notion of ``confinement'' for vertices following a correct triangle-vertex.

\needspace{2cm}

\begin{proposition}\label{prop:correct-forbidden-indices-property}
    Let $u$ be a correct triangle-vertex of $P$ of rank $a$, and $b$ be its index-lock.
    Then the following two assertions hold: 
    \begin{itemize}
        \item $\suffix{P}{u}$ does not contain a vertex of rank $b$ in $\zone_u(b)$; and
        \item $\suffix{P}{u}$ does not contain a triangle-vertex of rank $i\in \reach^+_b(a)$ in $\zone_u(b)$, and in particular not a rib of that rank contained in $\zone_u(b)$.
    \end{itemize}
\end{proposition}

\begin{proof}
    If $b=\ell+1$ then the first assertion trivially holds as there are no vertices of that rank in $G$.
    Let us thus prove the first assertion assuming $b\leq \ell$, i.e., there is a triangle-vertex of rank $b$ before $u$ in $\triangleseq(P)$ with $u$ in its zone. 
    By construction, the triangle $K$ at the root of this zone $\zone_u(b)$ sends ribs toward each blocking-vertex of rank $b$ in this zone.
    Moreover, since $b>a$ and so $u\not\in V(K)$, we note that for any vertex $v$ in $\prefix{P}{u}$, we have that $\subpath{P}{v}{u}$ contains vertices that are not in $K$.
    Hence $\suffix{P}{u}$ does not intersect $K$, as otherwise we would get a chord, which is excluded.
    
    Suppose toward a contradiction that $\suffix{P}{u}$ contains a vertex $v$ of rank $b$ in $\zone_u(b)$.
    Let $w$ be the vertex of $K$ that is contained in $P$.
    By the above discussion, $v$ does not lie in $K$.
    So $v$ must be the bottom endpoint of a rib of rank~$b$.
    Note that this rib cannot be an edge of $P$ as we would again induce a chord with $w$.
    Now since the endpoints of $P$ are triangle-vertices, we derive that $v$ must have degree two in $P$. 
    We conclude that $P$ must contain all the representatives of the occurrence of $b$, i.e., the three vertices $x_1,x_2,x_3$ that were created during the blow-up operation to represent this occurrence of the index $b$ in the barrier; see Section~\ref{sec:blow-up}.
    However one of these vertices will induce a chord with $w$, a contradiction.
    This concludes the proof of the first assertion.

    Let us now assume toward a contradiction that $P$ contains a triangle-vertex $v$ of rank $i\in \reach^+_b(a)$ contained in $\zone_u(b)$, with possibly $b=\ell+1$ (in which case $\zone_u(b)=V(G)$).
    As $u$ is correct, every such $i$ is burned, i.e., either $i$ is a ghost rank, or $\prefix{P}{u}$ contains the endpoint of a rib of index $i$ incident to some $w\in \ancestor_u(i)$.
    Clearly, $i$ is not a ghost rank as $v$ belongs to $P$.
    Thus the second property holds.
    Note that $\prefix{P}{u}$ either contains such a rib of rank $i$, or, using again the fact that the endpoints of $P$ are triangle-vertices, it contains the three representatives of the occurrence of $i$ in the barrier.
    Both possibilities induced a chord with $v$, a contradiction.
    This proves the second assertion and hence the proposition.
\end{proof}

As already mentioned in the proof of Proposition~\ref{prop:index-trace-factor}, note that after visiting a triangle-vertex $u$, by construction, the path $P$ has several options we review here.
Recall that hopping ribs are those additional ribs added between triangles of rank $1$ and $j\geq 1$; see Figures~\ref{fig:node-barrier} and \ref{fig:blowup}.
After visiting a triangle-vertex $u$, either:
\begin{itemize}
    \item $P$ visits a second triangle-vertex within the same triangle, in which case we say that $u$ is \new{staying};
    \item $P$ follows an edge from a barrier-path separating $K^{u}$ from its parent triangle, in which case we say that $u$ is \new{ascending};
    \item $P$ follows an edge from a barrier-path separating $K^{u}$ from a child triangle, in which case we say that $u$ is \new{descending};
    \item $P$ follows a rib $e$ incident to $u$, in which case we say that $u$ is \new{jumping}.
\end{itemize}
Moreover, when $u$ is ascending toward a triangle of rank $1$ it could be that $P$ contains a hopping rib of the same rank as $u$, whose group of representatives are encountered first in the barrier, in which case we say that $u$ is \emph{hopping up}\mgnt{hopping up/down}.
Similarly, when $u$ is jumping, it could be that the used rib $e$ is a hopping rib, in which case we say that $u$ is \emph{hopping down}.
We distinguish these two cases as they result into $P$ visiting two distinct triangles of a same rank, which, while irrelevant as far as the asymptotic order of a path is concerned, needs extra care in the coming statements.

An example of a triangle-vertex vertex $u$ that is ascending is given in Figure~\ref{fig:case-1} (right); one that is descending is given in Figure~\ref{fig:case-2} (right), and one that is hopping down is given in Figure~\ref{fig:case-3b-ii-iii} (left). 

To this terminology, we add one more notion, saying that a vertex $u$ is \new{locking} if $\next(u)\in\zone(u)$.
Note that a vertex $u$ may be locking while being descending or jumping, but not while being ascending, nor hopping up or down.
Moreover, we have the following by construction.

\begin{observation}\label{obs:lock-next-rank-smaller}
    If $v$ is locking then $\rank(\next(v))<\rank(v)$.
\end{observation}

The first of our two key lemmas is the following, which consists of a case study of the different behaviors of our path.

\begin{lemma}\label{lem:next-is-correct}
    Let $u$ be a correct triangle-vertex of $P$ of rank $a$, and $b$ be its index-lock.
    Let us further assume that $v:=\next(u)$ exists, and let $c$ be its rank. 
    Then $v$ lies in $\zone_v(b)$ and is correct.
    
    \medskip 
    
    \noindent{}Moreover, either $a=c<b$, in which case exactly one of the following holds:
    \begin{enumerate}[label=(\alph*)]
        \item\label{it:hoppingup} $u$ is hopping up, $\trace(\subpath{P}{u}{v})=\{a\}$, and $v$ is not hopping up;
        \item\label{it:hoppingdown} $u$ is hopping down, $\trace(\subpath{P}{u}{v})=\{a\}$, and $v$ is not hopping up nor down;
    \end{enumerate}
    or $c \reach_b a$, in which case exactly one of the following holds:
    \begin{enumerate}[label=(\alph*), resume]
        \item\label{it:ascending} $u$ is ascending, $c<a<b$, and $\trace(\subpath{P}{u}{v})=B(a,c)$;
        \item\label{it:descending} $u$ is either descending or jumping, $a<c<b$, and $\trace(\subpath{P}{u}{v})$ is the reverse of $B(c,a)$; moreover, $v$ can neither be ascending, nor hopping up or down;  
        \item\label{it:lock} $u$ is locking, $c<a<b$, and $\trace(\subpath{P}{u}{v})$ is a factor of $B(\ell)$ from an occurrence of $a$ to an occurrence of $c$.
    \end{enumerate}
\end{lemma}

\begin{proof}
    Let $K$ be the triangle in which $u$ belongs.
    Note that $u$ may be staying, ascending, descending, or jumping.
    Let us assume that it is not staying, as otherwise we can replace $u$ by the next vertex in $K$ (at most once) and conduct the same analysis.
    We will distinguish the remaining cases in the rest of the proof.
    For each case, we shall show $v\in \zone_v(b)$, and use Proposition~\ref{prop:correct-forbidden-indices-property} to show that $\prefix{P}{v}$ has visited blocking-vertices of each of the indices required to be burned for $v$ to be correct.
    Moreover, along the way, we show that $\subpath{P}{u}{v}$ is of one of the five types as described in the statement.

    In the rest of the proof, we will focus on the way $P$ intersects right barriers, and more particularly their blocking-vertices.
    We note however that left and right barriers make no distinction in the analysis.
    Moreover, by Proposition~\ref{prop:index-trace-factor}, we know that the traces of these intersections define factors of the corresponding index-barriers. 
    We shall use this observation implicitly, to show that the ranks ``reachable'' by $P$ in these barriers are those in $B(\ell)$.

    Finally, note that the different cases described in the statement are mutually exclusive: this is clear between Cases~\ref{it:hoppingup}, \ref{it:hoppingdown}, and \ref{it:descending}, as well as between Cases~\ref{it:ascending} and \ref{it:lock}, by their definition, 
    and it follows by the stated conditions on the ranks $a,c$ for the other comparabilities.
    Hence we shall only argue which case is fulfilled in the remaining of the proof.
    
    \medskip
    
    \needspace{1cm}\casestylei{Case 1:} $u$ is ascending; see Figure~\ref{fig:case-1}.
    
    Let $K'$ be the parent triangle of $K$.
    Note that since $b>a$ and the rank of $K'$ is one more than the rank of $K$, we have that the vertices in $K$, $K'$, and those contained in the barrier separating these triangles, all lie in $\zone_u(b)$.
    Hence Proposition~\ref{prop:correct-forbidden-indices-property} applies to these vertices, and we derive that $P$ cannot reach a vertex of rank $b$ in the barrier, nor use an attached rib of rank in $\reach^+_b(a)$.
    
    By construction, two cases arise depending on the rank of $K'$.

    \medskip
    
    \needspace{1cm}\casestyleii{Case 1.a:} the rank of $K'$ is $a+1$; see Figure~\ref{fig:case-1} (left).
    
    Then the trace of the barrier from $K$ to $K'$ is the reverse of $B(a+1,a)$.
    By Observation~\ref{obs:i-j-intermediate} the internal indices in this trace are all greater than $a$.
    Thus $P$ cannot use any rib attached to the barrier.
    So $P$ must reach $K'$ and we derive $a+1\in \reach_b^+(a)$.
    This contradicts Proposition~\ref{prop:correct-forbidden-indices-property} and we conclude that this case cannot arise.

    \begin{figure}[htb]
        \begin{minipage}[b]{0.48\textwidth}
            \includegraphics[scale=1,page=1]{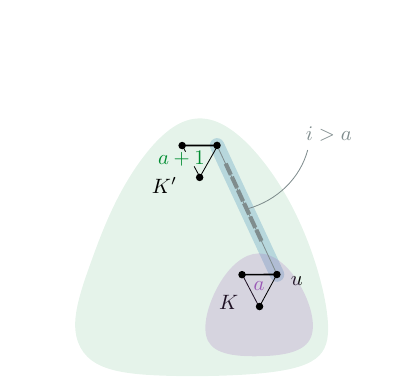}
        \end{minipage}
        \hspace{-40pt}
        \begin{minipage}[b]{0.48\textwidth}
            \includegraphics[scale=1,page=1]{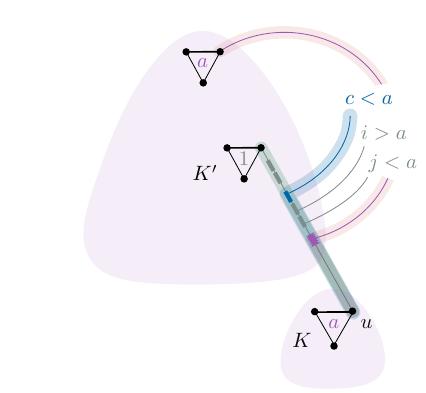}
        \end{minipage}
        \caption{The situations of Case 1.a leading to a contradiction (left) and of Cases 1.b.i (right, in green and blue) and 1.b.ii (right, in red).}
        \label{fig:case-1}
    \end{figure}

    \medskip

    \needspace{1cm}\casestyleii{Case 1.b:} the rank of $K'$ is $1$; see Figure~\ref{fig:case-1} (right).
    
    Then the trace of the barrier from $K$ to $K'$ is $B(a,1)$.
    Thus, either $P$ reaches $K'$, in which case $b$ does not appear in $B(a,1)$, or it contains a rib of rank in $\reach^-_b(a)$, or it contains a rib of rank $a$.

    \medskip

    \needspace{1cm}\casestyleiii{Case 1.b.i:} $P$ reaches $K'$, i.e., $c=1$, or contains a rib of rank $c\in \reach^-_b(a)$.
    
    Note that since $a,c<b$, $v\in \zone_u(b)$.
    Also, $c$ is the first occurrence of its index in the trace of $\subpath{P}{u}{v}$: this is by definition of the barrier if $c=1$, and in the other case, if $P$ has crossed another occurrence of $c$ before, then one of its attached rib would induce a chord with $v$, a contradiction.
    So this trace is a prefix of $B(a,1)$ from $a$ to the first occurrence of $c$.
    Hence it is equal to $B(a,c)$ by definition.
    By Observation~\ref{obs:i-j-intermediate}, this trace contains all indices from $a$ to $c$. 
    Note that $a$ appears as well as the rank of a blocking-vertex in the barrier because of the hopping rib of rank $a$.
    So each of these indices in burned by $v$.
    As $a,c<b$, by the transitivity of $\reach_b$, we have $\reach_b(c)=\reach_b(a)$.
    Moreover every $i\in\reach^+_b(a)$ is burned by $u$. 
    So every $i\in\reach^+_b(c)$ is burned by $v$ and we conclude that $v$ is correct.
    Note that this situation satisfies the conditions of Case~\ref{it:ascending}.

    \medskip
    
    \needspace{1cm}\casestyleiii{Case 1.b.ii:}
    $P$ contains a rib of rank $a$.
    
    Note that this rib must be a hopping rib as otherwise $P$ induces a chord with a hopping rib of the barrier. So $a=c<b$ and $\trace(\subpath{P}{u}{v})=\{a\}$. 
    Moreover, $v$ may not be hopping up once more, by construction.
    Hence this situation satisfies the conditions of Case~\ref{it:hoppingup}.

    \medskip

    \needspace{1cm}\casestylei{Case 2:}
    $u$ is descending; see Figure~\ref{fig:case-2}.
    
    By construction, the vertices contained in this barrier, as well as those in $K$, lie in $\zone_u(b)$; note that we do not (yet) assume that the same holds for the vertices in $K'$.
    Hence the conditions of Proposition~\ref{prop:correct-forbidden-indices-property} are satisfied for these vertices, and we derive that $P$ cannot reach a vertex of rank $b$ in the barrier, nor use an attached rib of rank in $\reach^+_b(a)$.  
    By construction, two cases arise depending on the value of $a$, i.e., the rank of $u$.

    \begin{figure}[htb]
        \hspace{-15pt}
        \begin{minipage}[b]{0.48\textwidth}
            \includegraphics[scale=1,page=2]{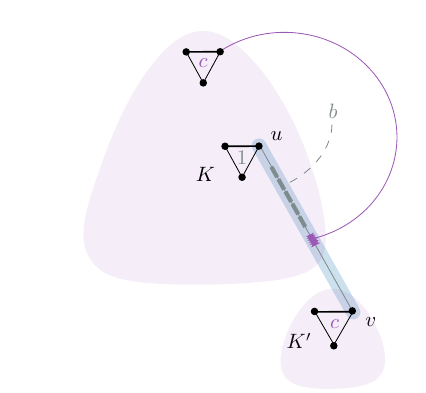}
        \end{minipage}
        \hspace{-10pt}
        \begin{minipage}[b]{0.48\textwidth}
            \includegraphics[scale=1,page=2]{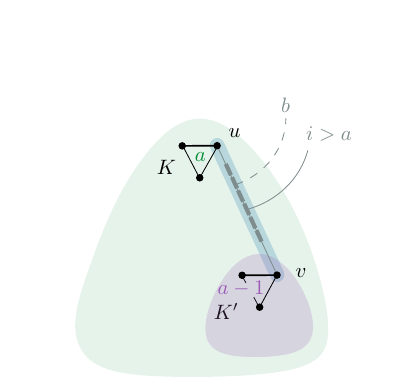}
        \end{minipage}
        \caption{The situations of Case 2.a (left) and Case 2.b (right) where $P$ (in blue) is forced to reach the child triangle $K'$.}
        \label{fig:case-2}
    \end{figure}

    \medskip
    
    \needspace{1cm}\casestyleii{Case 2.a:} ${a=1}$; Figure~\ref{fig:case-2} (left).
    
    Then the trace of the barrier from $u$ to $v$ is the reverse of $B(c,1)$.
    By Observation~\ref{obs:i-j-intermediate} this barrier contains indices from $c$ to $1$, and blocking vertices of rank $2$ to $c$, the later being included because of the hopping rib.
    Note that these indices are greater than $1$.
    Thus no rib can be used by $P$ while crossing this barrier.
    So $v$ must belong to $K'$, and we deduce that $b$ does not appear in the barrier.
    So $b>c$ and $v$ lies in $\zone_u(b)$.
    Since $b>c$ and $b>1$, by the transitivity of $\reach_b$, we have $\reach_b(1)=\reach_b(c)$.
    Now since all $i\in\reach^+_b(1)$ have been burned by $u$, we trivially get that every $i\in\reach^+_b(c)$ is burned by $v$ as well.
    So $v$ is correct.
    
    Let us argue that $v$ may not be ascending in that case.
    Note that since $P$ has crossed the right barrier separating $K$ and $K'$, the only way for $v$ to be ascending is for $P$ to now cross the left barrier. 
    But by the above arguments, no rib can be used by $P$ while crossing this barrier, and if $P$ reaches $K$ then it would induce a chord, a contradiction.
    
    Similarly, the fact that $v$ is not hopping follows by the fact the concerned hopping ribs of rank $c$ have been intersected by $P$ while crossing the barrier.
    
    Hence $u$ satisfies the conditions of Case~\ref{it:descending}.

    \medskip
    
    \needspace{1cm}\casestyleii{Case 2.b:} $a\geq 2$; see Figure~\ref{fig:case-2} (right).
    
    The trace of the barrier is ${B(a,a-1)}$.
    By Observation~\ref{obs:i-iminusone-intermediate}, its blocking-vertices are of rank greater than $a$.
    Hence the indices visited by $P$ belong to $\reach^+_b(a)$. 
    So no rib can be used by $P$ while crossing this barrier, and hence $v$ must belong to $K'$.
    Thus $v\in\zone_u(b)$.
    Since $v\in \zone(u)$, $u$ is locking and $a$ is the index-lock of $v$.

    Let us show that every index in $\reach^+_a(a-1)$ is visited by $P$; we stress the fact that we are now dealing with $a$ as a blocking index, and not $b$.
    By Proposition~\ref{prop:same-subtree}, $a-1$ either lies in $\ct(a^-)$, or in $\ct(k^+)$ for some ancestor $k$ of $a$ with possibly $k=a$.
    Moreover by Proposition~\ref{prop:from-parent-of-root-to-i-consecutive}, all indices in $\reach_a^+(a-1)$ are visited, as the trace of $\subpath{P}{u}{v}$ contains such an ancestor by Observation~\ref{obs:factors-contain-sp}.
    Hence $v$ is correct. 
    
    The trace of $\subpath{P}{u}{v}$ is equal to ${B(a,a-1)}$, hence is a factor of $B(\ell)$.
    Since $u$ is locking the conditions of Case~\ref{it:lock} are fulfilled.
       
    \medskip

    \needspace{1cm}\casestylei{Case 3:} $u$ is jumping; see Figures~\ref{fig:case-3a}--\ref{fig:case-3b-ii-iii}.
    
    Note that $u$ is the top endpoint of a rib having its bottom endpoint in a barrier separating two triangles.
    Let $K^s$ and $K^t$ be these triangles corresponding to two adjacent nodes $s,t$ of respective rank $i,j$ with $s$ is the parent of $t$.
    Note that the trace of the barrier from $s$ to $t$ is either $B(i,j)$ if $i\geq 2$, in which case $j=i-1$, or the reverse of $B(j,1)$, if $i=1$.

    \begin{figure}[htb]
        \hspace{10pt}
        \begin{minipage}[b]{0.48\textwidth}
            \includegraphics[scale=1,page=3]{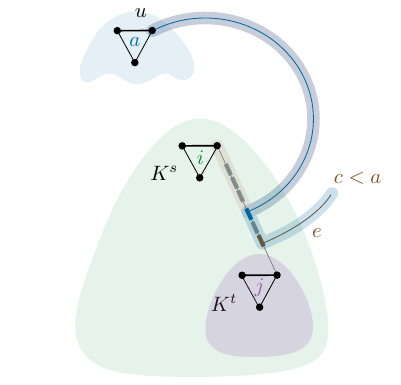}
        \end{minipage}
        \hspace{-60pt}
        \begin{minipage}[b]{0.48\textwidth}
            \includegraphics[scale=1,page=3]{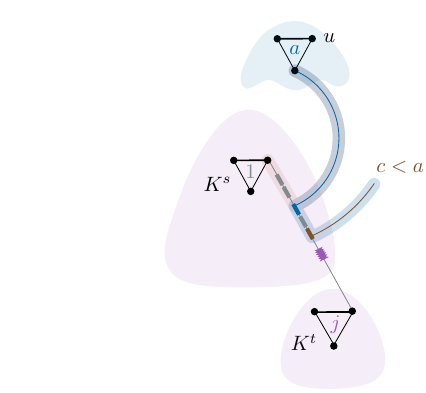}
        \end{minipage}
      \caption{The situations of Case 3.a when $i=j+1$ (left, in red if $P$ reaches $K^s$, in blue otherwise) and $i=1$ (right, in red if $P$ reaches $K^s$, in blue otherwise).}
      \label{fig:case-3a}
    \end{figure}
    
    By construction, since a rib incident to $u$ intersects this barrier, the vertices in $K^s$ and those in the barrier belong to $\zone(u)$, hence to $\zone_u(b)$.
    Thus the conditions of Proposition~\ref{prop:correct-forbidden-indices-property} are satisfied for these vertices, and we derive that $P$ cannot reach a vertex of rank $b$ in the barrier, nor reach a triangle or use an attached rib of rank in $\reach^+_b(a)$.
    It can neither use a rib of rank $a$ as taking these ribs would induce a chord with $u$.
    Thus $P$ can either reach $K^s$, use a rib in the barrier, or reach $K^t$, if their corresponding rank $c$ lies in $\reach^-_b(a)$.
    We distinguish these cases.

    \medskip

    \needspace{1cm}\casestyleii{Case 3.a:} $P$ either contains a rib or reaches $K^s$; see Figure~\ref{fig:case-3a}.
    
    The coming arguments hold in these two situations regardless of whether  $i=1$ or $i\geq 2$.
    As the vertices of $K^s$ and those in the barrier belong to $\zone(u)$, we have $c<a$ and $v\in\zone_b(u)$.
    Since $c<a$, we have $v=\next(u)\in \zone(u)$ and so $u$ is locking.
    Let us show that every index in $\reach^+_a(c)$ is visited by $P$; note that we are now dealing with $a$ as a blocking index, and not $b$.
    Since $c<a$, by Proposition~\ref{prop:same-subtree}, $c$ either lies in $\ct(a^-)$, or in $\ct(k^+)$ for some ancestor $k$ of $a$ with possibly $k=a$.
    Moreover by Proposition~\ref{prop:from-parent-of-root-to-i-consecutive}, all indices in $\reach_a^+(c)$ are visited, as the trace of $\subpath{P}{u}{v}$ contains such an ancestor by Observation~\ref{obs:factors-contain-sp}.
    Thus $v$ is correct. 

    The trace of $\subpath{P}{u}{v}$ is a factor of $B(i,j)$, hence of $B(\ell)$.
    Since $u$ is locking the conditions of Case~\ref{it:lock} are fulfilled.

    \medskip

    \begin{figure}[htb]
        \hspace{15pt}
        \begin{minipage}[b]{0.48\textwidth}
        \includegraphics[scale=1, page=4]{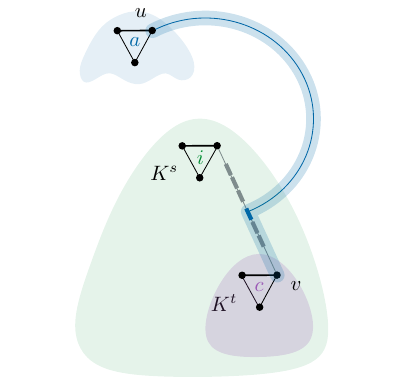}
        \end{minipage}
        \hspace{-70pt}
        \begin{minipage}[b]{0.48\textwidth}
        \includegraphics[scale=1, page=4]{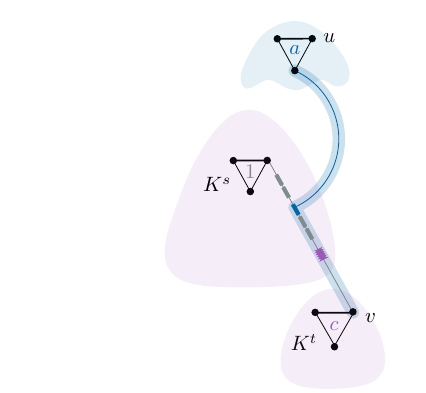}
        \end{minipage}
      \caption{The situations of Case 3.b.i when $i\geq 2$ (left) and $i=1$ (right).} 
      \label{fig:case-3b-i}
    \end{figure}

    \needspace{1cm}\casestyleii{Case 3.b:} $P$ reaches $v$ in $K^t$; see Figures~\ref{fig:case-3b-i}--\ref{fig:case-3b-ii-iii}.
    
    Thus we have $c=j$ in the remaining.
    Two situations arise depending on whether $i\geq 2$, in which case $j=i-1$, or if $i=1$, in which case $j\geq 1$.
    If $i\geq 2$, then $v$ also lies in $\zone_u(b)$.
    The same analysis as in the previous cases yields the desired conclusion with $u$ locking, each of $i\in \reach_a^+(c)$ visited by $P$, $v$ correct, and the trace of $\subpath{P}{u}{v}$ being a factor of $B(\ell)$ from $a$ to $c$, i.e., Case~\ref{it:lock}.
    
    Otherwise, if $i=1$, we distinguish three cases based on the values of $a$ and $c$, namely, depending on whether we have $c<a$, $c=a$, or $c>a$. 

    \medskip
    
    \needspace{1cm}\casestyleiii{Case 3.b.i:} $c<a$; see Figure~\ref{fig:case-3b-i} (right).
    
    The trace of $\subpath{P}{u}{v}$ is the reverse of a prefix of $B(c,1)$ going from the starting occurrence of $c$ to the first occurrence of $a$. 
    Hence this trace is equal to the reverse of $B(c,a)$ by definition.
    Since $c<a$, $v$ belongs to $\zone(u)$ and so $u$ is locking. 
    Let us show that every index in $\reach^+_a(c)$ is visited by $P$.
    Since $c<a$, by Proposition~\ref{prop:same-subtree}, $c$ either lies in $\ct(a^-)$, or in $\ct(k^+)$ for some ancestor $k$ of $a$ with possibly $k=a$.
    Moreover by Proposition~\ref{prop:from-parent-of-root-to-i-consecutive}, all indices in $\reach_a^+(c)$ are visited, as the trace of $\subpath{P}{u}{v}$ contains such an ancestor by Observation~\ref{obs:factors-contain-sp}.
    Thus $v$ is correct.
    
    Note that the trace of $\subpath{P}{u}{v}$ is a factor from $c$ to $a$.
    Hence it satisfies the conditions of Case~\ref{it:lock}.
    
    \medskip

    \begin{figure}[htb]
        \hspace{-20pt}
        \begin{minipage}[b]{0.48\textwidth}
        \includegraphics[scale=1, page=5]{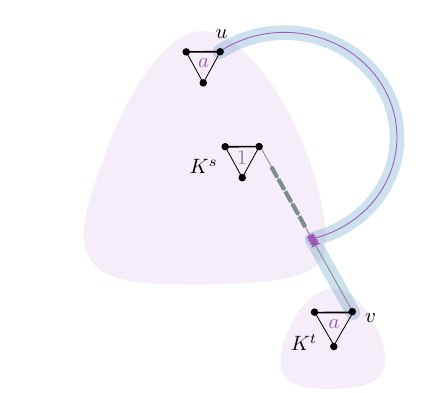}
        \end{minipage}
        \hspace{-10pt}
        \begin{minipage}[b]{0.48\textwidth}
        \includegraphics[scale=1, page=6]{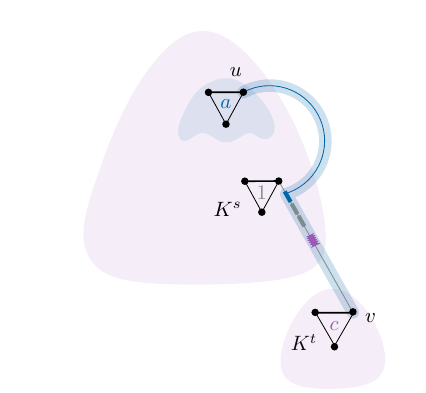}
        \end{minipage}
      \caption{The situations of Case 3.b.i and Case 3.b.iii.} 
      \label{fig:case-3b-ii-iii}
    \end{figure}
    
    \needspace{1cm}\casestyleiii{Case 3.b.ii:} $a=c$; see Figure~\ref{fig:case-3b-ii-iii} (left). 
    
    The rib incident to $u$ in $P$ must be a hopping rib, as otherwise the hopping ribs of the barrier would induce a chord with $P$. 
    As $u$ is correct and $a=c$, we trivially get that $v$ is correct.
    
    Note that the trace of $\subpath{P}{u}{v}$ is $\{a\}$.
    By construction, $v$ cannot be hopping down, as such ribs are only between consecutive zones of a same rank. Moreover $v$ cannot be hopping up as $P$ would induce a chord with $u$.
    This corresponds to Case~\ref{it:hoppingdown}.
    
    \medskip
    
    \needspace{1cm}\casestyleiii{Case 3.b.iii:}
    $c>a$; see Figure~\ref{fig:case-3b-ii-iii} (right).
    
    Then the trace of $\subpath{P}{u}{v}$ is the reverse of a prefix of $B(c,1)$ going from the starting occurrence of $c$ to the first occurrence of $a$. 
    Hence this trace is equal to the reverse of $B(c,a)$ by definition. 
    By Observation~\ref{obs:i-j-intermediate} this barrier contains blocking-vertices of ranks from $c-1$ to $a+1$, to which we add $c$ because of the hopping rib.
    Moreover $b>c$ as otherwise, since $b>a$, we would have $c>b>a$ and by Observation~\ref{obs:i-j-intermediate} the barrier would contain $b$, a contradiction.
    Since $b>c,a$, by the transitivity of $\reach_b$, we have $\reach_b(a)=\reach_b(c)$.
    Now since all $i\in\reach^+_b(a)$ have been burned by $u$, clearly every $i\in\reach^+_b(c)$ is burned by $v$.
    So $v$ is correct.

    Let us argue that $v$ may not be ascending in that case.
    Note that since $P$ has crossed a part of the right barrier separating $K^s$ and $K^t$, the only way for $v$ to be ascending is for $P$ to now cross the left barrier. 
    However, on this barrier, $P$ is stopped by the rib incident to $u$.
    As of the ribs that precede it, they cannot be used since their corresponding rank has been burned while reaching $K^t$.
    Hence $v$ is not ascending.

    Similarly, $v$ may not be hopping up since the concerned hopping ribs are those of rank $c$, but they have been intersected by $P$ while crossing the right barrier.
    Finally, there is no hopping rib incident to $v$ by construction.
    
    Hence $u$ satisfies the conditions of Case~\ref{it:descending}.

    \medskip

    This concludes the case study and hence the proof.
\end{proof}

We derive the following statement as a corollary of Lemma~\ref{lem:next-is-correct} by induction on the sequence of (representatives of) triangle-vertices in $P$.

\begingroup

\begin{corollary}\label{cor:lock-decrease-rank-suffix}
    Every triangle-vertex of $P$ is correct.
    Moreover, if $u$ is locking, then $\rank(v)<\rank(u)$ for every vertex $v$ after $u$ in $\triangleseq(P)$.
\end{corollary}

\begin{proof}
    The first assertion follows by induction on $\triangleseq(P)$, noting that the first triangle-vertex of $P$ is trivially correct, and that the next triangle-vertex of a correct vertex in $\triangleseq(P)$ is correct by Lemma~\ref{lem:next-is-correct}.
    
    We turn to proving the second part of the statement. 
    Note that if $u$ is locking then, by Observation~\ref{obs:lock-next-rank-smaller}, $\rank(v)<\rank(u)$ for $v:=\next(u)$, and $v$ has $\rank(u)$ for index-lock.
    Now, since every vertex is correct as proven above, by Lemma~\ref{lem:next-is-correct}, $\next(v)$ has its index-lock which may only decrease compared to that of $v$, and this is true for any subsequent vertex in $\triangleseq(P)$.
    We conclude by Observation~\ref{obs:index-lock-rank-decreases} noting that the rank of a vertex is bounded by its index-lock. 
\end{proof}

\endgroup

Among the vertices in $\triangleseq(P)$, let us define \new{$\specialseq(P)$} as its subsequence starting with the first triangle-vertex (which is of largest rank), containing locking vertices, and ending with the last triangle-vertex.
Again, let us point that a same triangle of $G$ may be intersected twice by $P$, and that only one representative is selected in $\specialseq(P)$, to avoid unnecessary technicalities.

Note that vertices in $\specialseq(P)$ have in common that the vertices that follow them in $P$, if they exist, are of a rank smaller than their index-lock: this is trivial for the vertex of largest rank as its index-lock is $\ell+1$, and it follows from Corollary~\ref{cor:lock-decrease-rank-suffix} for the others.
We shall in fact show a stronger statement, using the following notion.
Recall that $\ct(i)$ denotes the subtree of $\ct$ rooted at index $i$.

\begin{definition}[Range]
    The \emph{range}\mgnt{range, $\range(v)$} of a triangle-vertex $v$, denoted $\range(v)$, is $\ct(k)$ for the largest $k$ such that $\rank(v)\in \ct(k)$ and $b\not\in \ct(k)$, for $b$ its index-lock.
\end{definition}

Note that the range of the first vertex $v$ in $\specialseq(P)$ is thus $\ct$, and the same holds true for vertices of largest rank following $v$ in $\triangleseq(P)$.
 
We show that the index $k$ yielding the range $\ct(k)$ is characterized as the root of $\ct$ or of a subtree as described in Proposition~\ref{prop:smaller-index-incomparable-right-child}.

\begin{proposition}\label{prop:range-root-is-smaller}
    Let $v$ be a triangle-vertex of $\specialseq(P)$.
    If $k$ is such that $\ct(k)=\range(v)$, then $k<b$ for $b$ its index-lock.
    Moreover, either $k=\ell$, $k=b^-$, or $k=j^+$ for some ancestor $j$ of $b$ with possibly $j=b$. 
\end{proposition}

\begin{proof}
    Consider a vertex $v$ in $\specialseq(P)$.
    By Observation~\ref{obs:index-lock-rank-decreases}, $\rank(v)<b$.
    If $b=\ell+1$ the statement holds for $k=\ell$.
    Otherwise, by~Proposition~\ref{prop:smaller-index-incomparable-right-child}, $\rank(v)$ either lies in $\ct(b^-)$, or it lies in $\ct(j^+)$ for some ancestor $j$ of $b$ with possibly $j=b$.
    Note that $\range(v)$ must be disjoint from the path from $b$ to the root of $\ct$. 
    Hence $k\neq j$ and we derive that $k=b^-$ or $k=j^+$ since such a $k$ is maximal with the property that $\rank(v)\in \ct(k)$ and $b\not\in \ct(k)$.
\end{proof}

We are now ready to prove the second of our two key lemmas, which we will use to show that the first triangle-vertex and locking vertices confine the remaining of the (trace of the) path into further and further nested subtrees of $\ct$.

\needspace{2cm}

\begingroup
\newcommand\nextu{u'}
\begin{lemma}\label{lem:between-locks}
    Let $u,v$ be two consecutive vertices in $\specialseq(P)$, and $\nextu:=\next(u)$.
    Then $\subpath{P}{\nextu}{v}$ only contains triangle-vertices whose ranks are in~$\range(v)$, and all these vertices have the same range. 
    Moreover, the length of $\subpath{P}{\nextu}{v}$ is bounded by $10\cdot |B(k)|$, for $k$ such that $\ct(k)=\range(v)$. 
\end{lemma}

\begin{proof}
    Consider the sequence $\nextu=v_1,\dots,v_p=v$ of consecutive triangle-vertices of $\triangleseq(P)$ in $\subpath{P}{u'}{v}$.    
    Let $b$ be the index-lock of $u$.
    By definition of $\specialseq$, either $b=\ell+1$, if $u$ is of largest rank, or $b\leq \ell$.
    Also, note that we may have ${\rank(v_1)<\rank(u)}$, or ${\rank(v_1)=\rank(u)}$ and $u$ is not locking and so it must be of largest rank in $P$.

    We distinguish this latter case where $\rank(v_1)=\rank(u)$.
    Then $v_1$ is also of largest rank.
    Then from the different possibilities that Lemma~\ref{lem:next-is-correct} offers, together with the fact that $u$ and $v_1$ are of largest rank, it must be that $u$ is hopping up, or down.
    See Figure~\ref{fig:largest-ranks} for an illustration of this case.
    So $v_1$ is also of largest rank.
    Then $v_1$ may be locking, in which case we have $v=v_1$, or $v_2$ is also of largest rank, in which case we must have that $u$ is hopping up, and $v_1$ hopping down, by the combinations allowed by Lemma~\ref{lem:next-is-correct}.
    However in that latter case, $v_2$ must be locking, and so $v=v_2$.
    Note that the bound is satisfied for these two distinct cases, since $|B(k)|\geq 1$ and $P$ visits at most two triangles before reaching $v$, each being intersected twice.
    Therefore the range of these vertices is $\ct$ and we conclude to the desired claim.

    Thus, we may now assume that $v_1$ is such that $\rank(v_1)<\rank(u)$.
    Note that we may still have $b=\ell+1$ or $b\leq \ell$.
    
    By Lemma~\ref{lem:next-is-correct}, every $v_i$, $1\leq i<p$ either verify $\rank(v_i)=\rank(v_{i+1})$ or it verifies $\rank(v_{i+1})\in \reach_b(\rank(v_i))$, with $\rank(v_{i+1})<b$ in both cases.
    Hence each of $v_1,\dots,v_p$ has its rank smaller than $b$.
    By the transitivity of $\reach_b$ we derive the equality $\reach_b(\rank(v_i))=\reach_b(\rank(v_{i+1}))$ for all such $i$. 
    If $b\leq \ell$, by Lemma~\ref{prop:same-subtree}, the ranks of $v_1,\dots,v_p$ all lie in a single subtree of $\ct$ being either $\ct(b^-)$, or $\ct(j^+)$ for some ancestor $j$ of $b$ with possibly $j=b$.
    This subtree is precisely the one having $k$ as its root by Proposition~\ref{prop:range-root-is-smaller}, as $T(k)=\range(v)$.
    Otherwise if $b=\ell+1$ then the range is $\ct(k)$ for $k=\ell$.
    So the ranks of $v_1,\dots,v_p$ all lie in $\ct(k)$.

    Let us now consider the shape of $v_1,\dots,v_p$.
    By definition, no vertex in this sequence, except possibly for $v_1$ and $v_p$, are locking.
    By Lemma~\ref{lem:next-is-correct}, four possibilities arise:
    \begin{itemize}
        \item $v_i$ is hopping up, in which case $v_{i+1}$ is not;
        \item $v_i$ is hopping down, in which case $v_{i+1}$ is neither hopping up nor down;
        \item $v_i$ is ascending and satisfies $\rank(v_{i+1})<\rank(v_i)$;
        \item $v_i$ is descending or jumping, with $\rank(v_{i+1})>\rank(v_i)$; moreover in that case, $v_{i+1}$ can neither be ascending, hopping up, nor hopping down.
    \end{itemize}
    From the different combinations these possibilities allow, we derive that $P$ may possibly be ascending several times---while hopping up and/or hopping down once at each step---and then it is only descending/jumping.
    Note that the maximum value of $p$ is thus achieved when $P$ has ascended as much as possible, i.e., $\rank(v)\leq k$ times since these ranks lie in $\ct(k)$, and then has descended as much as possible, i.e., $k$ times. 
    We thus get that $p\leq 3k + k= 4k$.

    Let us now focus on the trace of $v_1,\dots,v_q$.
    By Lemma~\ref{lem:next-is-correct}, the trace of $P$ between consecutive $v_i,v_{i+1}$ is equal to $B(\rank(v_i),\rank(v_{i+1}))$ when $v_i$ is ascending, to $B(\rank(v_{i+1}),\rank(v_i))$ when it is descending/jumping, or it consists of the single index $\rank(v_i)$ otherwise.
    Thus, in addition to the above bound on $p$, from the characterization of the shape of $v_1,\dots,v_q$ we derive the following.
    When ascending, the concatenation of $B(\rank(v_i),\rank(v_{i+1}))$ for consecutive $v_i$ defines a factor of $B(\ell)$, hence of $B(k)$, up to contracting repeated indices by Observation~\ref{obs:concatenation-of-barriers}.
    We derive the same when $P$ is descending/jumping. 

    So the trace of $P$ from $u$ to $v$ has total length bounded by $2\cdot |B(k)|$.
    Note that the repeated indices, in the concatenation, are those of the triangle-vertices visited by $P$.
    Hence their number is bounded by $p$.
    Now, since to each index in a barrier corresponds $3$ vertices representing it, we derive that $P$ has length at most $3\cdot 2\cdot |B(k)|+p=6\cdot |B(k)| + 4k\leq 10\cdot |B(k)|$.
\end{proof}

\endgroup 

\begin{figure}[htb]
    \centering
    \includegraphics[scale=1, page=7]{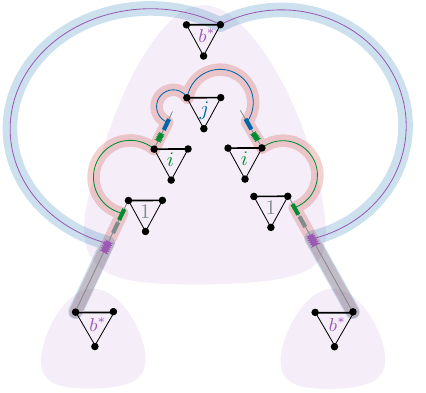}
    \caption{Two possible situations for $P$ containing vertices of largest rank $b^*$. In red, two consecutive vertices or largest rank in $\specialseq(P)$, separated by a sequence of ascending and descending/jumping vertices as detailed in the proof of Lemma~\ref{lem:between-locks}. In blue, three consecutive vertices of largest rank, the first one being hopping up, the second being hopping down.} 
    \label{fig:largest-ranks}
\end{figure}

We conclude this section with the following bound on the order of an induced path which has triangle-vertices as its endpoints, with the first endpoint being of largest rank.
This bound is what will be used to prove our main theorem.

\needspace{2cm}

\begin{lemma}\label{lem:total-length}
    Let $u_1,\dots,u_p$ be the vertices of $\specialseq(P)$ and $k_1,\dots,k_p$ be the indices such that $\ct(k_i)=\range(u_i)$ for every $1\leq i\leq p$.
    Then 
    $\ct(k_{i+1})$ is a proper subtree of $\ct(k_i)$ for all $2\leq i<p$, 
    and the order of $P$ is bounded by
    \[
        13\cdot \left(\sum_{i=1}^{p} |B(k_i)|\right)+1.
    \]
\end{lemma}

\begin{proof}
    We first show that $\ct(k_{i+1})$ is a subtree of $\ct(k_i)$ for all $2\leq i<p$.
    Note that $u_i$ is locking by definition and the choice of $i$.    
    Let $a$ be its rank, and let $v:=\next(u_i)$ and $c:=\rank(v)$.
    Then $c<a$.

    By Lemma~\ref{lem:next-is-correct}, $c\in \reach_b(a)$, with $b$ the index-lock of $u_i$.
    If $b\leq \ell$, by Proposition \ref{prop:same-subtree}, $a$ and $c$ both lie in $\ct(k)$ for $k$ being either equal to $b^-$, or to $j^+$ for some ancestor $j$ of $b$ with possibly $j=b$. 
    If $b=\ell+1$, they trivially both lie in $T(k)$ for $k=\ell$.
    By Proposition~\ref{prop:range-root-is-smaller} we actually have $k=k_i$.
    Moreover by Lemma~\ref{lem:between-locks} the range of $v$ is $\ct(k_{i+1})$.
    Since $v$ lies in $\ct(k_{i+1})$ and $b$ does not, $\ct(k_{i+1})$ must be a subtree of $\ct(k_i)$.
    Since $a$ is the index-lock of $c$, $\ct(k_{i+1})$ does not contain $a$.
    We conclude that $\ct(k_{i+1})$ is a proper subtree of $\ct(k_i)$, as desired.

    We now prove the bound on the order of $P$.
    Let $v_1,\dots,v_{p-1}$ be such that $v_i:=\next(u_i)$ for all $i\in\intv{1}{p-1}$.
    Note that $P$ is the concatenation of 
    $\subpath{P}{u_i}{v_i}$ and of $\subpath{P}{v_i}{u_{i+1}}$ for $i$ ranging from $1$ to $p-1$, plus possibly one extra vertex which may be the one sharing the same triangle as $u_{p}$.
    
    We start with paths of the first type.
    Note that $\subpath{P}{u_1}{v_1}$ is limited to a single barrier, so it is of order bounded by $3\cdot |B(\ell)|$, hence bounded by $3\cdot |B(k_1)|$ since the range of $u_1$ is $\ct$.
    
    For $i\geq 2$ we have that $u_i$ is locking.
    In that case, by Lemma~\ref{lem:next-is-correct}, we have that the trace of $\subpath{P}{u_i}{v_i}$ is limited to a factor of $B(\ell)$ going from an occurrence of $\rank(u_i)$ to an occurrence of $\rank(v_i)$.
    Moreover by the same lemma we have $\rank(v_i)\in \reach_{b_i}(\rank(u_i))$ for $b_i$ being the index lock of $u_i$.
    Hence both $\rank(u_i)$ and $\rank(v_i)$ lie in the range $\ct(k_i)$ of $u_i$.
    So this factor is bounded by $3\cdot |B(k_i)|$.

    Finally, by Lemma~\ref{lem:between-locks}, the order of each of the subpaths $\subpath{P}{v_i}{u_{i+1}}$ is bounded by $10\cdot |B(k_{i+1})|$.
    Finally, the last vertex counts for one.
    We conclude to the desired bound.
\end{proof}

\section{Size of the barriers}\label{sec:barrier-size}

As we saw in the previous section, up to fixing an orientation and assuming endpoints to be triangle-vertices with the first one of largest rank---two assumptions that we will later argue to be reasonable---the length of longest induced paths in our construction can be bounded from above by a function of the size of the index-barriers. 
Recall that an index-barrier is constructed from an index-tree (see Definition~\ref{def:full-barrier}). From the definition it can be observed that the size of the index-barrier is strongly related to the two following parameters of the index-tree: the size of the left subtree, which is explored twice, and the depth of the tree.

An easy recurrence\footnote{As we will not use this result, we refrain from giving a proof.} shows that if $\ct$ is a complete binary tree, then the length of the full barrier $B(\ell)$ is bounded by $\ell^{\log 3}$ which is $O(\ell^{1.6})$.
We show here that taking for $\ct$ an appropriately unbalanced tree, the bound become $O(\ell \log \ell)$.

\newcommand{\n}{m}
\needspace{2cm}

\begin{definition}
    Given an integer $\n$ and a real $\alpha>2$, we define a tree  \new{$T_\alpha(\n)$} inductively as follows:
    \begin{itemize}
        \item If $\n=0$, $T_\alpha(\n)$ is empty;
        \item Otherwise, it consists of a root with $T_{\alpha(\alpha-1)}(\lfloor \n/\alpha\rfloor)$ as its left child, and $T_{\alpha}(\n-\lfloor \n/\alpha\rfloor-1)$ as its right child.
    \end{itemize}
\end{definition}

Note that a node in this tree can have no left child and still have a right child.
The following is a consequence of the definition.

\begin{lemma}\label{lem:imbalanced-index-tree-size}
    $T_\alpha(\n)$ contains $\n$ nodes.
\end{lemma}

Let us denote by \new{$B_\alpha( \n)$} the full index-barrier of Definition~\ref{def:full-barrier} with $\ct:=T_\alpha(m)$ as an index-tree. (For simplicity, let us assume that the full index-barrier of the empty tree is defined as the empty word, and that the full index-barrier of $i$ with a single child is defined by $B(i)=i\circ i \circ B(i-1) \circ i \circ i$, i.e., as if $i-1$ was a right child. Note that we can make these assumptions as it upper bounds the size of the actual barrier which is defined as $i \circ B(i-1) \circ i$ in that case.)

\begin{lemma}\label{lem:imbalanced-barrier-size}
     For every integer $\n$ and real $\alpha >2$,
    the index-barrier $B_\alpha(\n)$ has length at most 
    $4\cdot \frac{\n\log (\alpha \n)}{\log (\alpha-1)}.$
\end{lemma} 

\begin{proof}
    The proof is by induction on $\n$. The statement is true for $\n=0$ and $\n=1$ as the barrier has respectively length 0 or 1 in these cases.
    So we now consider $\n>1$ and suppose that the statement is true for all $\n'<\n$ and all $\alpha'>2$ (induction hypothesis).
    By definition of the full index-barrier (Definition~\ref{def:full-barrier}) we have: 
    \begingroup
    \allowdisplaybreaks 
    \begin{align*}
        |B_\alpha(n)| &= 2|B_{\alpha(\alpha-1)}(\lfloor \n/\alpha\rfloor)| + |B_\alpha(\n-\lfloor \n/\alpha\rfloor-1)| + 4\\
        &\le 
            2\cdot 4\cdot \frac{\lfloor \n/\alpha\rfloor\log \big(\alpha(\alpha-1)\lfloor \n/\alpha\rfloor\big)}{\log\big(\alpha(\alpha-1)-1\big)}\\
            &\phantom{\le 2}\hspace{-.077cm} 
            + 4\cdot \frac{(\n-\lfloor \n/\alpha)\rfloor-1)\cdot\log\big(\alpha(\n-\lfloor \n/\alpha)\rfloor-1)\big)}{\log(\alpha-1)}
            +4\\
            &\le 
            2\cdot 4\cdot \frac{\lfloor \n/\alpha\rfloor\log (\alpha \n)}{\log\big((\alpha-1)^2+\alpha-2\big)}
            +4\cdot \frac{(\n-\lfloor \n/\alpha)\rfloor-1)\cdot\log(\alpha \n)}{\log(\alpha-1)}
            +4\\
            &\le 
            2\cdot 4\cdot \frac{\lfloor \n/\alpha\rfloor\log (\alpha \n)}{\log\big((\alpha-1)^2\big)}
            +4\cdot \frac{(\n-\lfloor \n/\alpha)\rfloor-1)\cdot\log(\alpha \n)}{\log(\alpha-1)}
            +4\\
            &\le 
             4\cdot \frac{\lfloor \n/\alpha\rfloor\log (\alpha \n)}{\log(\alpha-1)}
            +4\cdot \frac{(\n-\lfloor \n/\alpha)\rfloor-1)\cdot\log(\alpha \n)}{\log(\alpha-1)}
            +4\\
            &\le 
             4\cdot \frac{\log \alpha \n}{\log(\alpha-1)}\cdot(\lfloor \n/\alpha\rfloor + \big(\n-\lfloor \n/\alpha)\rfloor-1)\big) + 4\\
            &\le 
             4\cdot \frac{\log (\alpha \n)}{\log(\alpha-1)}\cdot \n 
            =
             4\cdot \frac{\n\log (\alpha \n)}{\log(\alpha-1)}
             \qedhere
    \end{align*}
    \endgroup
\end{proof} 

We conclude this section by showing that the sum of Lemma~\ref{lem:total-length} is similarly bounded under the condition that subtrees are properly included.

\begin{lemma}\label{lem:imbalanced-sum-of-barriers-size}
    Given an integer $\n$ and a real $\alpha >2$, let $\ct:=T_\alpha(\n)$ and $k_1,\dots,k_p$ be any sequence of indices of $\ct$ such that $\ct(k_{i+1})$ is a proper subtree of $\ct(k_i)$ for all $i\in\intv{1}{p-1}$. 
    Then
    \[
        \sum_{i=1}^{p} |B(k_i)| \leq 4\cdot \frac{\n\log(\alpha \n)\alpha}{\log(\alpha-1)}.
    \]
\end{lemma} 

\begin{proof}
\begingroup
\allowdisplaybreaks 
    We proceed by induction on $\n$. The inequality is vacuously true for $\n=0$ as $\ct$ is empty.
    Let us first show that the inequality holds for $\n=1$.
    In that case, note that $\ct$ has one node, and hence $p=1$.
    By Lemma~\ref{lem:imbalanced-barrier-size} we have $|B(k_1)|\leq 4\cdot (\n\log (\alpha \n))/(\log (\alpha-1))$ which is less than the desired bound, so the statement holds. 
    
    Suppose now that $\n > 1$ and that the statement holds for any $\n'<\n$ and $\alpha'>2$.
    We distinguish cases depending whether $k_2$ lies in the subtree rooted at the left or right child of the root of $\ct$. In each case we use Lemma~\ref{lem:imbalanced-barrier-size} to bound the size of $B(k_1)$ and the induction hypothesis for the rest of the sum, i.e., $\sum_{i=2}^p |B(k_i)|$.
    When $k_2$ is in the subtree rooted at the left child of the root of $\ct$, we have:
    \begin{align*}
        \sum_{i=1}^p |B(k_i)|&\le|B(k_1)|+\sum_{i=2}^p |B(k_i)|\\
        &\le 4\cdot \frac{\n\log(\alpha \n)}{\log(\alpha-1)}
        + 4\cdot \frac{\lfloor \n/\alpha\rfloor\log\big(\alpha(\alpha -1) \lfloor \n/\alpha\rfloor\big)(\alpha(\alpha-1))}{\log\big(\alpha(\alpha- 1)-1\big)}\\
        &\le 4\cdot \left(\frac{\n\log(\alpha \n)}{\log(\alpha-1)}+\frac{\n\log(\alpha \n)(\alpha(\alpha-1))}{\alpha\cdot\log(\alpha(\alpha- 1)-1)}\right)\\
        &\le 4\cdot \left(\frac{\n\log( \alpha \n)}{\log(\alpha-1)}+\frac{\n\log(\alpha \n)(\alpha-1)}{\log(\alpha- 1)}\right)\\
        &\le 4\cdot \frac{\n\log(\alpha \n)}{\log(\alpha-1)}\cdot (1+\alpha-1)\\
        &\le 4\cdot \frac{\n\log (\alpha \n)\alpha }{\log(\alpha-1)}.
    \end{align*}
    Similarly when $k_2$ lies in the right subtree we get:
    \begin{align*}
        \sum_{i=1}^p |B(k_i)|&\le|B(k_1)|+\sum_{i=2}^p |B(k_i)|\\
        &\le 4\cdot \left(\frac{\n\log(\alpha \n)}{\log(\alpha-1)}+\frac{(\n-\lfloor \n/\alpha\rfloor-1)\log(\alpha\cdot (\n-\lfloor \n/\alpha\rfloor-1))\alpha}{\log(\alpha-1)}\right)\\
        &\le 4\cdot \left(\frac{\n\log(\alpha \n)}{\log(\alpha-1)}+\frac{\n(1-1/\alpha)\log(\alpha \n)\alpha}{\log(\alpha-1)}\right)\\
        &\le4\cdot \frac{\n\log(\alpha \n)}{\log(\alpha-1)}\cdot(1+(1-1/\alpha)\alpha)\\
        &\le4\cdot \frac{\n\log(\alpha \n)\alpha}{\log(\alpha-1)}.
    \end{align*}
    Hence in both cases we get the desired bound, and by induction the result holds for every $m$ and $\alpha$ as claimed.
\endgroup
\end{proof}

\section{The proof of the main theorem}\label{sec:proof}

We are now ready to prove Theorem~\ref{thm:main-theorem}, that we restate below for convenience.

\restatemaintheorem*

\begin{proof}
Let $\ell\geq 3$ be some integer. We set $\ct:=T(3,\ell)$, $G:=G(\ct)$, and $n:=2^{2^\ell}$.
By Remark~\ref{rem:graph-size}, $G$ has $\Omega(n)$ vertices and by Lemma~\ref{lem:traceable} it has a Hamiltonian path, so it has a path of order $\Omega(n)$. By lemma~\ref{lem:degenerate}, $G$ is $2$-degenerate. It remains to show that $G$ does not have long induced paths.

Let $Q$ be an induced path of $G$ and consider a subpath $P$ of $Q$ of maximum order that starts from a triangle-vertex of largest rank in $Q$ and ends with a triangle-vertex.
Note that the order of $Q$ is bounded by twice the order of $P$, plus twice the size of $B(\ell)$, the latter part being a bound on the order of the extremities of $Q$ before the first triangle-vertex, and after the last.
Indeed, by construction, these extremities must lie in a single barrier each, as they do not contain any triangle-vertex.

Let $v_1,\dots,v_p$ be the sequence of vertices in $\specialseq(P)$, as defined in Section~\ref{sec:path-properties}.
Let $k_1,\dots,k_p$ be the indices such that $\ct(k_i)=\range(v_i)$ for every such $i$.
By Lemma~\ref{lem:total-length} we have that $\ct(k_{i+1})$ is a proper subtree of $\ct(k_i)$ for every $2\leq i\leq p$, and the total length of $P$ is bounded by 
\[
    13\cdot \left(\sum_{i=1}^{p} |B(k_i)|\right) +1 \leq 13\cdot\left( |B(\ell)| + \sum_{i=2}^{p} |B(k_i)|\right) + 1.
\]
Applying Lemma~\ref{lem:imbalanced-barrier-size} for the size of $B(\ell)$, and Lemma~\ref{lem:imbalanced-sum-of-barriers-size} for the right hand sum, with $\n=\ell$ and $\alpha=3$, the above is upper bounded by $624 \ell \log \ell$.
%

Applying Lemma~\ref{lem:imbalanced-barrier-size} with the same parameters for each of the two extremities of $Q$, we overall get that $Q$ has order at most $936 \cdot \ell \log \ell = O(\log\log n \cdot \log \log \log n)$, as desired.
This concludes the proof.
\end{proof}

\begin{remark}
We note that the asymptotic dependence in the size of $B(\ell)$ for the order of a longest induced path in our construction is tight, as an induced path may traverse a barrier corresponding to $B(\ell-1,1)$ between a node of rank $\ell-1$ and its parent of rank $1$. 
Indeed, this barrier has order $\Omega(|B(\ell)|)$.
However, it is not clear whether our choice of $\ct$ results in the smallest barrier, or if the value of $B(\ell)$ is $\Omega(\ell \log \ell)$ for our choice of $\ct$.
\end{remark}

\bibliographystyle{alpha}
\bibliography{main}

@book{nevsetvril2012sparsity,
  title={Sparsity: graphs, structures, and algorithms},
  author={Ne{\v{s}}et{\v{r}}il, Jaroslav and Ossona de Mendez, Patrice},
  volume={28},
  year={2012},
  publisher={Springer Science \& Business Media}
}

@article{duron2024long,
  title={Long induced paths in sparse graphs and graphs with forbidden patterns},
  author={Duron, Julien and Esperet, Louis and Raymond, Jean-Florent},
  journal={arXiv preprint arXiv:2411.08685},
  year={2024}
}

@article{hunter2024long,
  title={Long induced paths in {$K_{s,s}$-free} graphs},
  author={Hunter, Zach and Milojevi{\'c}, Aleksa and Sudakov, Benny and Tomon, Istv{\'a}n},
  journal={arXiv preprint arXiv:2411.19173},
  year={2024}
}

@article{galvin1982ramsey,
  title={A {R}amsey-type theorem for traceable graphs},
  author={Galvin, Fred and Rival, Ivan and Sands, Bill},
  journal={Journal of Combinatorial Theory, Series B},
  volume={33},
  number={1},
  pages={7--16},
  year={1982},
  publisher={Elsevier}
}

@article{defrain2024sparse,
  title={Sparse graphs without long induced paths},
  author={Defrain, Oscar and Raymond, Jean-Florent},
  journal={Journal of Combinatorial Theory, Series B},
  volume={166},
  pages={30--49},
  year={2024},
  publisher={Elsevier}
}

@article{ding1996unavoidable,
  title={Unavoidable minors of large 3-connected binary matroids},
  author={Ding, Guoli and Oporowski, Bogdan and Oxley, James and Vertigan, Dirk},
  journal={journal of combinatorial theory, Series B},
  volume={66},
  number={2},
  pages={334--360},
  year={1996},
  publisher={Elsevier}
}

@inproceedings{atminas2012linear,
  title={Linear time algorithm for computing a small biclique in graphs without long induced paths},
  author={Atminas, Aistis and Lozin, Vadim V and Razgon, Igor},
  booktitle={Scandinavian Workshop on Algorithm Theory},
  pages={142--152},
  year={2012},
  organization={Springer}
}

@phdthesis{allred2022unavoidable,
  title={Unavoidable Structures in Large and Infinite Graphs},
  author={Allred, Sarah},
  year={2022},
  school       = {Louisiana State University and Agricultural \& Mechanical College},
  type         = {PhD thesis}
}

\end{document}